\documentclass[a4paper,12pt]{article}

\usepackage{authblk}
\usepackage{parskip}
\usepackage{mathtools}
\usepackage{amssymb}
\usepackage{amsmath}
\usepackage{latexsym}
\usepackage{mathrsfs}  
\usepackage{bbm}       
\usepackage{amsthm}    
\usepackage{relsize}   
\usepackage{graphicx}
\usepackage{thmtools}  
\usepackage{mathrsfs}  
\usepackage{paralist}  
\usepackage{lipsum}    
\usepackage[all]{xy}   

\numberwithin{equation}{section}

\usepackage{titlesec}   

\titleformat{\section}[block]{\bfseries\filcenter}
{{\upshape\thesection\enspace}}{.5em}{}

\titleformat{\subsection}[block]{\filcenter}
{{\upshape\thesubsection\enspace}}{.5em}{} 

\titleformat{\subsubsection}[block]{\filcenter}
{{\upshape\thesubsubsection\enspace}}{.5em}{} 

\usepackage{enumitem}  
\setlist{nosep}  
\setitemize[0]{leftmargin=*}
\setenumerate[0]{leftmargin=*}
\setenumerate[1]{label={(\arabic*)}} 

\newcommand{\N}{\mathbb{N}}     
\newcommand{\R}{\mathbb{R}}     
\newcommand{\C}{\mathbb{C}}     
\newcommand{\Prob}{\mathbb{P}}  
\newcommand{\Exp}{\mathbb{E}}   
\newcommand{\goth}[1]{\mathfrak{#1}} 
\newcommand{\ind}[2]{\mathbbm{1}_{#1}\left( #2 \right)}          
\newcommand{\inner}[2]{\left\langle #1 \, , \, #2 \right\rangle} 
\newcommand{\norm}[1]{\left|\left|#1\right|\right|}              
\newcommand{\triplet}[3]{\left( #1, #2, #3 \right) }             
\newcommand{\ProbSpace}{\triplet{\Omega}{\mathscr{F}}{\Prob}}    
\newcommand{\abs}[1]{\left| #1 \right|}                          
\renewcommand{\qedsymbol}{$\square$}                       

\newcommand{\defeq}{\mathrel{\mathop:}=}                         
\newcommand\restr[2]{{
  \left.\kern-\nulldelimiterspace 
  #1 
  \vphantom{\big|} 
  \right|_{#2} 
  }}

\makeatletter
\newsavebox{\@brx}
\newcommand{\llangle}[1][]{\savebox{\@brx}{\(\m@th{#1\langle}\)}%
  \mathopen{\copy\@brx\kern-0.5\wd\@brx\usebox{\@brx}}}
\newcommand{\rrangle}[1][]{\savebox{\@brx}{\(\m@th{#1\rangle}\)}%
  \mathclose{\copy\@brx\kern-0.5\wd\@brx\usebox{\@brx}}}
\makeatother

\theoremstyle{plain} 
\newtheorem{theo}{Theorem}[section]    
\newtheorem{prop}[theo]{Proposition} 
\newtheorem{coro}[theo]{Corollary}
\newtheorem{lemm}[theo]{Lemma}
\newtheorem{assu}[theo]{Assumption}

\theoremstyle{definition} 
\newtheorem{defi}[theo]{Definition}
\newtheorem{exam}[theo]{Example}

\newtheorem{rema}[theo]{Remark}

\declaretheoremstyle[%
  spaceabove=-5pt,%
  spacebelow=6pt,%
  headfont=\normalfont\itshape,%
  postheadspace=1em,%
  qed=\qedsymbol%
]{mystyle} 
\declaretheorem[name={Proof},style=mystyle,unnumbered,
]{prf}

 \title{Semimartingales on Duals of Nuclear Spaces}
\author{C. A. Fonseca-Mora}
\affil{  Escuela de Matem\'{a}tica, Universidad de Costa Rica, \\ San Jos\'{e}, 11501-2060, Costa Rica. 

\noindent E-mail:  christianandres.fonseca@ucr.ac.cr }

\date{}

\begin{document}

 \maketitle

\abstract{
This work is devoted to the study of semimartingales on the dual of a general nuclear space. We start by establishing conditions for a cylindrical semimartingale in the strong dual $\Phi'$ of a nuclear space $\Phi$ to have a $\Phi'$-valued semimartingale version  whose paths are right-continuous with left limits. Results of similar nature but for more specific classes of cylindrical semimartingales and examples are also provided.   
Later, we will show that under some general conditions every semimartingale taking values in the dual of a nuclear space has a canonical representation. The concept of predictable characteristics is introduced and is used to establish necessary and sufficient conditions for a $\Phi'$-valued semimartingale to be a $\Phi'$-valued L\'{e}vy process. 
}

\smallskip



\section{Introduction}

In recent years, there has been an increasing interest in the study of cylindrical stochastic processes in infinite dimensional spaces. One of the main motivations is the use of these random objects to construct stochastic integrals and as the driving noise to a stochastic partial differential equation (see e.g. \cite{ApplebaumRiedle:2010, FonsecaMora:2018-1, IssoglioRiedle:2014, PriolaZabczyk:2011,  Riedle:2015, Riedle:2018, VeraarYaroslavtsev:2016}).
Due to the great importance of the semimartingales in the theory of stochastic calculus, it is only natural to consider cylindrical semimartingales as the driving force for these types of stochastic equations. 

Let $\Phi$ be a nuclear space and let $\Phi'$ its strong dual. Under the  assumption that $\Phi'$ is a complete nuclear space, the concept of $\Phi'$-valued semimartingales was firstly introduced by \"{U}st\"{u}nel in \cite{Ustunel:1982}. There, \"{U}st\"{u}nel used the fact that under these assumptions $\Phi'$ can be expressed as a projective limit of Hilbert spaces with Hilbert-Schmidt embeddings and defined the $\Phi'$-valued semimartingales as a projective system of Hilbert space valued semimartingales. Further properties of $\Phi'$-valued semimartingales and stochastic calculus with respect to them where explored  by  \"{U}st\"{u}nel in his subsequent works \cite{Ustunel:1982-1, Ustunel:1984, Ustunel:1986}, by P\'{e}rez-Abreu \cite{PerezAbreu:1988}, and by other authors (e.g. \cite{BojdeckiGorostiza:1991, DawsonGorostiza:1990,  PerezAbreuRochaArteagaTudor:2005}). However, \"{U}st\"{u}nel's approach for semimartingales as projective systems can no longer be applied if one assumes that $\Phi$ is a general nuclear space, because in that case $\Phi'$ is usually not nuclear nor complete. Therefore, if $\Phi$ is assumed only to be nuclear, a new theory of semimartingales has to be developed and this is exactly the purpose of this paper; to begin a systematic study of cylindrical semimartingales and semimartingales in the strong dual $\Phi'$ of a general nuclear space $\Phi$. Our main motivation is the further introduction of a theory of stochastic integration and of stochastic partial differential equations driven by semimartingales in duals of nuclear spaces. This program was carried out by the author in \cite{FonsecaMora:2018-1} for the L\'{e}vy case, but for the semimartingale case a new approach has to be considered. The necessary properties of semimartingales are developed in this paper.  The corresponding application to stochastic integration and existence of solutions to stochastic partial differential equations will appear in forthcoming papers.  

We start in Sect. \ref{secNotaDefi} by setting our notation and by reviewing some useful properties of nuclear spaces, cylindrical processes, and the space $S^{0}$ of real-valued semimartingales. Then, in Sect. \ref{secCylSemDualNuclear} we address the first main problem of our study; it consists in to establish conditions under which a cylindrical semimartingale $X=( X_{t}: t \geq 0) $ in $\Phi'$ does have a semimartingale version with c\`{a}dl\`{a}g (i.e. right-continuous with left limits) paths. In the literature this procedure is often known as ``regularization'' (see \cite{FonsecaMora:2018, Ito, PerezAbreu:1988}). Our main result is Theorem \ref{theoRegulCylinSemimartigalesNuclear} where it is shown that a sufficient condition is equicontinuity for each $T>0$ of the family $( X_{t}: t \in [0,T])$ as operators from $\Phi$ into the space of real-valued random variables $L^{0}\ProbSpace$ defined on a given probability space $\ProbSpace$. 
In Proposition \ref{propCylSemimContOpeSpaceSemim} we show that this condition is equivalent to the condition that $X$ as a linear map from $\Phi$ into  $S^{0}$ be continuous. As a consequence of our result we show that if the nuclear space is ultrabornological (e.g. Fr\'{e}chet), then each $\Phi'$-valued semimartingale with Radon laws has a  c\`{a}dl\`{a}g version. Regularization theorems for more specific classes of semimartingales and examples are also considered. Our regularization results generalize those obtained by \"{U}st\"{u}nel in \cite{Ustunel:1982, Ustunel:1982-1} and by P\'{e}rez-Abreu \cite{PerezAbreu:1988} where only $\Phi'$-valued semimartingales (but not cylindrical semimartingales) were considered and it is assumed that $\Phi'$ is complete nuclear. 

Our second main aim on this work is to find a canonical representation for $\Phi'$-valued semimartingales. This is done in Sect. \ref{secDecomSemi} by using our regularization results from Sect. \ref{secCylSemDualNuclear}. Our main result is Theorem \ref{theoSemirtDecompNucleSpace} where we show that if for a $\Phi'$-valued semimartingale we assume equicontinuity for each $T>0$ of the induced cylindrical process $( X_{t}: t \in [0,T])$, then $X$ possesses a canonical representation similar to that for semimartingales in finite dimensions. This equicontinuity assumption is always satisfied if the nuclear space is ultrabornological and the semimartingale have Radon laws. Furthermore, we introduce the concept of predictable characteristics for $\Phi'$-valued semimartingales. To the extend of our knowledge the only previous work that considered the  existence of a canonical representation for semimartingales in the dual of a nuclear Fr\'{e}chet space was carried out by P\'{e}rez-Abreu in \cite{PerezAbreu:1988}. Observe that our result generalize that of P\'{e}rez-Abreu to semimartingales in the dual of a general nuclear space. 

Finally, in Sect. \ref{sectCharLevyProce} we examine in detail the canonical representation of a $\Phi'$-valued L\'{e}vy process and its relation with their L\'{e}vy-It\^{o} decomposition studied by the author in \cite{FonsecaMora:Levy}. It is shown that the predictable characteristics of a L\'{e}vy process coincide with those of its L\'{e}vy-Khintchine formula and that the particular form of these characteristics distinguish the L\'{e}vy process among the $\Phi'$-valued semimartingales. 
  
We hope that our semimartingale canonical representation and our definition of characteristics could be used in the future to study further properties of $\Phi'$-valued semimartingales, as for example to study functional central limit theorems.

\section{Definitions and Notation}\label{secNotaDefi}


Let $\Phi$ be a locally convex space (we will only consider vector spaces over $\R$). 
If $p$ is a continuous semi-norm on $\Phi$ and $r>0$, the closed ball of radius $r$ of $p$ given by $B_{p}(r) = \left\{ \phi \in \Phi: p(\phi) \leq r \right\}$ is a closed, convex, balanced neighborhood of zero in $\Phi$. A continuous semi-norm (respectively a norm) $p$ on $\Phi$ is called \emph{Hilbertian} if $p(\phi)^{2}=Q(\phi,\phi)$, for all $\phi \in \Phi$, where $Q$ is a symmetric, non-negative bilinear form (respectively inner product) on $\Phi \times \Phi$. Let $\Phi_{p}$ be the Hilbert space that corresponds to the completion of the pre-Hilbert space $(\Phi / \mbox{ker}(p), \tilde{p})$, where $\tilde{p}(\phi+\mbox{ker}(p))=p(\phi)$ for each $\phi \in \Phi$. The quotient map $\Phi \rightarrow \Phi / \mbox{ker}(p)$ has an unique continuous linear extension $i_{p}:\Phi \rightarrow \Phi_{p}$.  Let $q$ be another continuous Hilbertian semi-norm on $\Phi$ for which $p \leq q$. In this case, $\mbox{ker}(q) \subseteq \mbox{ker}(p)$. Moreover, the inclusion map from $\Phi / \mbox{ker}(q)$ into $\Phi / \mbox{ker}(p)$ is linear and continuous, and therefore it has a unique continuous extension $i_{p,q}:\Phi_{q} \rightarrow \Phi_{p}$. Furthermore, we have the following relation: $i_{p}=i_{p,q} \circ i_{q}$. 

We denote by $\Phi'$ the topological dual of $\Phi$ and by $\inner{f}{\phi}$ the canonical pairing of elements $f \in \Phi'$, $\phi \in \Phi$. Unless otherwise specified, $\Phi'$ will always be consider equipped with its \emph{strong topology}, i.e. the topology on $\Phi'$ generated by the family of semi-norms $( \eta_{B} )$, where for each $B \subseteq \Phi$ bounded we have $\eta_{B}(f)=\sup \{ \abs{\inner{f}{\phi}}: \phi \in B \}$ for all $f \in \Phi'$.  If $p$ is a continuous Hilbertian semi-norm on $\Phi$, then we denote by $\Phi'_{p}$ the Hilbert space dual to $\Phi_{p}$. The dual norm $p'$ on $\Phi'_{p}$ is given by $p'(f)=\sup \{ \abs{\inner{f}{\phi}}:  \phi \in B_{p}(1) \}$ for all $ f \in \Phi'_{p}$. Moreover, the dual operator $i_{p}'$ corresponds to the canonical inclusion from $\Phi'_{p}$ into $\Phi'$ and it is linear and continuous. 

Let $p$ and $q$ be continuous Hilbertian semi-norms on $\Phi$ such that $p \leq q$.
The space of continuous linear operators (respectively Hilbert-Schmidt operators) from $\Phi_{q}$ into $\Phi_{p}$ is denoted by $\mathcal{L}(\Phi_{q},\Phi_{p})$ (respectively $\mathcal{L}_{2}(\Phi_{q},\Phi_{p})$) and the operator norm (respectively Hilbert-Schmidt norm) is denote by $\norm{\cdot}_{\mathcal{L}(\Phi_{q},\Phi_{p})}$ (respectively $\norm{\cdot}_{\mathcal{L}_{2}(\Phi_{q},\Phi_{p})}$). We employ an analogous notation for operators between the dual spaces $\Phi'_{p}$ and $\Phi'_{q}$. 

A locally convex space is called \emph{ultrabornological} if it is the inductive limit of a family of Banach spaces. A \emph{barreled space} is a locally convex space such that every convex, balanced, absorbing and closed subset is a neighborhood of zero. For equivalent definitions see \cite{Jarchow, NariciBeckenstein}. 
  
Let us recall that a (Hausdorff) locally convex space $(\Phi,\mathcal{T})$ is called \emph{nuclear} if its topology $\mathcal{T}$ is generated by a family $\Pi$ of Hilbertian semi-norms such that for each $p \in \Pi$ there exists $q \in \Pi$, satisfying $p \leq q$ and the canonical inclusion $i_{p,q}: \Phi_{q} \rightarrow \Phi_{p}$ is Hilbert-Schmidt. Other equivalent definitions of nuclear spaces can be found in \cite{Pietsch, Treves}. 

Let $\Phi$ be a nuclear space. If $p$ is a continuous Hilbertian semi-norm  on $\Phi$, then the Hilbert space $\Phi_{p}$ is separable (see \cite{Pietsch}, Proposition 4.4.9 and Theorem 4.4.10, p.82). Now, let $( p_{n} : n \in \N)$ be an increasing sequence of continuous Hilbertian semi-norms on $(\Phi,\mathcal{T})$. We denote by $\theta$ the locally convex topology on $\Phi$ generated by the family $( p_{n} : n \in \N)$. The topology $\theta$ is weaker than $\mathcal{T}$. We  will call $\theta$ a (weaker) \emph{countably Hilbertian topology} on $\Phi$ and we denote by $\Phi_{\theta}$ the space $(\Phi,\theta)$ and by $\widehat{\Phi}_{\theta}$ its completion. The space $\widehat{\Phi}_{\theta}$ is a (not necessarily Hausdorff) separable, complete, pseudo-metrizable (hence Baire and ultrabornological; see Example 13.2.8(b) and Theorem 13.2.12 in \cite{NariciBeckenstein}) locally convex space and its dual space satisfies $(\widehat{\Phi}_{\theta})'=(\Phi_{\theta})'=\bigcup_{n \in \N} \Phi'_{p_{n}}$ (see \cite{FonsecaMora:2018}, Proposition 2.4). 

\begin{exam}\label{examNuclearSpaces} It is well-known (see e.g. \cite{Pietsch, Schaefer, Treves}) that the space of test functions $\mathscr{E}_{K} \defeq \mathcal{C}^{\infty}(K)$ ($K$: compact subset of $\R^{d}$), $\mathscr{E}\defeq \mathcal{C}^{\infty}(\R^{d})$, the rapidly decreasing functions $\mathscr{S}(\R^{d})$, and the space of harmonic functions $\mathcal{H}(U)$ ($U$: open subset of $\R^{d}$; see \cite{Pietsch}, Section 6.3),  are all  examples of Fr\'{e}chet nuclear spaces. Their (strong) dual spaces $\mathscr{E}'_{K}$, $\mathscr{E}'$, $\mathscr{S}'(\R^{d})$, $\mathcal{H}'(U)$, are also nuclear spaces. On the other hand, the space of test functions $\mathscr{D}(U) \defeq \mathcal{C}_{c}^{\infty}(U)$ ($U$: open subset of $\R^{d}$), the space of polynomials $\mathcal{P}_{n}$ in $n$-variables, the space of real-valued sequences $\R^{\N}$ (with direct sum topology) are strict inductive limits of Fr\'{e}chet nuclear spaces (hence they are also nuclear). The space of distributions  $\mathscr{D}'(U)$  ($U$: open subset of $\R^{d}$) is also nuclear. All the above are examples of (complete) ultrabornological nuclear spaces. Other interesting examples of nuclear spaces are the space of continuous linear operators between a semi-reflexive dual nuclear space into a nuclear space, and tensor products of arbitrary nuclear spaces. These spaces need not be ultrabornological as for example happens with the space $\mathscr{D}(\R) \widehat{\otimes} \mathscr{S}(\R^{d}) $ (see \cite{BojdeckiGorostizaRamaswamy:1986}). 
\end{exam}


Throughout this work we assume that $\ProbSpace$ is a complete probability space and consider a filtration $( \mathcal{F}_{t} : t \geq 0)$ on $\ProbSpace$ that satisfies the \emph{usual conditions}, i.e. it is right continuous and $\mathcal{F}_{0}$ contains all subsets of sets of $\mathcal{F}$ of $\Prob$-measure zero. We denote by $L^{0} \ProbSpace$ the space of equivalence classes of real-valued random variables defined on $\ProbSpace$. We always consider the space $L^{0} \ProbSpace$ equipped with the topology of convergence in probability and in this case it is a complete, metrizable, topological vector space. 


A \emph{cylindrical random variable}\index{cylindrical random variable} in $\Phi'$ is a linear map $X: \Phi \rightarrow L^{0} \ProbSpace$ (see \cite{FonsecaMora:2018}). If $X$ is a cylindrical random variable in $\Phi'$, we say that $X$ is \emph{$n$-integrable} ($n \in \N$)  if $ \Exp \left( \abs{X(\phi)}^{n} \right)< \infty$, $\forall \, \phi \in \Phi$, and has \emph{zero-mean} if $ \Exp \left( X(\phi) \right)=0$, $\forall \phi \in \Phi$. The \emph{Fourier transform} of $X$ is the map from $\Phi$ into $\C$ given by $\phi \mapsto \Exp ( e^{i X(\phi)})$.

Let $X$ be a $\Phi'$-valued random variable, i.e. $X:\Omega \rightarrow \Phi'_{\beta}$ is a $\mathscr{F}/\mathcal{B}(\Phi'_{\beta})$-measurable map. For each $\phi \in \Phi$ we denote by $\inner{X}{\phi}$ the real-valued random variable defined by $\inner{X}{\phi}(\omega) \defeq \inner{X(\omega)}{\phi}$, for all $\omega \in \Omega$. The linear mapping $\phi \mapsto \inner{X}{\phi}$ is called the \emph{cylindrical random variable induced/defined by} $X$. We will say that a $\Phi'$-valued random variable $X$ is \emph{$n$-integrable} ($n \in \N$) if the cylindrical random variable induced by $X$ is \emph{$n$-integrable}. 
 
Let $J=\R_{+} \defeq [0,\infty)$ or $J=[0,T]$ for  $T>0$. We say that $X=( X_{t}: t \in J)$ is a \emph{cylindrical process} in $\Phi'$ if $X_{t}$ is a cylindrical random variable for each $t \in J$. Clearly, any $\Phi'$-valued stochastic processes $X=( X_{t}: t \in J)$ induces/defines a cylindrical process under the prescription: $\inner{X}{\phi}=( \inner{X_{t}}{\phi}: t \in J)$, for each $\phi \in \Phi$. 

If $X$ is a cylindrical random variable in $\Phi'$, a $\Phi'$-valued random variable $Y$ is called a \emph{version} of $X$ if for every $\phi \in \Phi$, $X(\phi)=\inner{Y}{\phi}$ $\Prob$-a.e. A $\Phi'$-valued process $Y=(Y_{t}:t \in J)$ is said to be a $\Phi'$-valued \emph{version} of the cylindrical process $X=(X_{t}: t \in J)$ on $\Phi'$ if for each $t \in J$, $Y_{t}$ is a $\Phi'$-valued version of $X_{t}$.  

For a $\Phi'$-valued process $X=( X_{t}: t \in J)$ terms like continuous, c\`{a}dl\`{a}g, purely discontinuous, adapted, predictable, etc. have the usual (obvious) meaning. 

A $\Phi'$-valued random variable $X$ is called \emph{regular} if there exists a weaker countably Hilbertian topology $\theta$ on $\Phi$ such that $\Prob( \omega: X(\omega) \in (\widehat{\Phi}_{\theta})')=1$. Furthermore, a $\Phi'$-valued process $Y=(Y_{t}:t \in J)$ is said to be \emph{regular} if $Y_{t}$ is a regular random variable for each $t \in J$. In that case the law of each $Y_{t}$ is a Radon measure in $\Phi'$ (see Theorem 2.10 in \cite{FonsecaMora:2018}).

If $\Psi$ is a separable Hilbert space, recall that a $\Psi$-valued adapted c\`{a}dl\`{a}g  process $X=( X_{t}: t \geq 0) $ is a \emph{semimartingale} if it admit a representation of the form 
$$X_{t}=X_{0} + M_{t}+A_{t}, \quad \forall \, t \geq 0, $$
where $M=(M_{t}:t \geq 0)$ is a c\`{a}dl\`{a}g local martingale and $A=(A_{t}: t \geq 0)$ is a c\`{a}dl\`{a}g adapted process of finite variation, and $M_{0}=A_{0}=0$. The reader is referred to \cite{Metivier} for the basic theory of Hilbert space valued semimartingales and to \cite{DellacherieMeyer, JacodShiryaev, Protter} for $\R^{d}$-valued semimartingales.  

We denote by $S^{0}$ the linear space of (equivalence classes) of real-valued semimartingales. We consider on $S^{0}$ the topology given in Memin \cite{Memin:1980}: define $\abs{\cdot}_{0}$ on  $S^{0}$ by  
$$ \abs{z}_{0}=\sum_{k=1}^{\infty} 2^{-k} \abs{z}_{k}, $$
where for each $k \in \N$, 
$$ \abs{z}_{k}=\sup_{h \in \mathcal{E}_{1}} \Exp  \left( 1 \wedge \abs{(h \cdot z)_{k} }\right), $$
$\mathcal{E}_{1}$ is the class of real-valued predictable processes $h$ of the form 
$$ h=\sum_{i=1}^{n-1} a_{i} \mathbbm{1}_{(t_{i}, t_{i+1}] \times \Omega},  $$ 
for $0 < t_{1} < t_{2} < \dots t_{n} < \infty$, $a_{i}$ is an $\mathcal{F}_{t_{i}}$-measurable random variable, $\abs{a_{i}} \leq 1$, $i=1, \dots, n-1$, and 
$$ (h \cdot z)_{k}=\int_{0}^{k} h_{s} dz_{s}= \sum_{i=1}^{n-1} a_{i} \left( z_{t_{i+1} \wedge k}-z_{t_{i} \wedge k} \right). $$  
Then, $d(y,z)=\abs{y-z}_{0}$ is a metric on $S^{0}$ and $(S^{0},d)$ is a complete, metric, topological vector space (is not in general locally convex). 
Unless otherwise specified, the space $S^{0}$ will always be consider equipped with this topology.  

If $x=(x_{t}:t \geq 0)$ is a real-valued semimartingale, we denote by $\norm{x}_{S^{p}}$ ($1 \leq p < \infty$) the following quantity: 
$$\norm{x}_{S^{p}} = \inf \left\{ \norm{ [m,m]_{\infty}^{1/2} + \int_{0}^{\infty} \abs{d a_{s} } }_{L^{p}\ProbSpace} : x=m+a \right\}, $$
where the infimum is taken over all the decompositions $x=m+a$ as a sum of a local martingale $m$ and a process of finite variation $a$. Recall that  $([m,m]_{t}: t \geq 0)$ denotes the quadratic variation process associated to the local martingale $m$, i.e. $[m,m]_{t}=\llangle  m^{c} , m^{c}  \rrangle_{t}+\sum_{0 \leq s \leq t} (\Delta m_{s})^{2}$, where $m^{c}$ is the (unique) continuous local martingale part of $m$ and $(\llangle  m^{c} , m^{c}  \rrangle_{t}:t \geq 0)$ its angle bracket process (see Section I in \cite{JacodShiryaev}). The set of all semimartingales $x$ for which $\norm{x}_{S^{p}}< \infty$ is a Banach space under the norm $\norm{\cdot}_{S^{p}}$ and is denoted by $S^{p}$ (see VII.98 in \cite{DellacherieMeyer}). Furthermore, if $x=m+a$ is a decomposition of $x$ such that $\norm{x}_{S^{p}}< \infty$ it is known that in such a case $a$ is of integrable variation (see VII.98(c) in \cite{DellacherieMeyer}). 

\section{Cylindrical Semimartingales in the Dual of a Nuclear Space} \label{secCylSemDualNuclear}

\begin{assu}
Throught this section and unless otherwise specified $\Phi$ will always denote a nuclear space. 
\end{assu}

\subsection{Regularization of Cylindrical Semimartigales}

In this section our main objective is to establish sufficient conditions for a cylindrical semimartingale in $\Phi'$  to have a $\Phi'$-valued c\`{a}dl\`{a}g semimartingale version. 
Following \cite{FonsecaMora:2018} and \cite{PerezAbreu:1988}, we call such a result as \emph{regularization of cylindrical semimartingales} (see Theorem \ref{theoRegulCylinSemimartigalesNuclear}). As a consequence of our results for cylindrical semimartingales we will show that if we assume some extra structure on the space $\Phi$ then every $\Phi'$-semimartingale has a c\`{a}dl\`{a}g version (see Proposition \ref{propCadlagVerSemimUltrabNuclear}).  We begin by introducing  our definition of cylindrical semimartingales in duals of nuclear spaces.  

\begin{defi}
A \emph{cylindrical semimartingale} in $\Phi'$ is a cylindrical process $X=(X_{t}: t \geq 0)$ in $\Phi'$ such that $\forall \phi \in \Phi$, $X(\phi)$ is a real-valued semimartingale. A \emph{cylindrical local martingale (resp. cylindrical martingale)} in $\Phi'$ is a cylindrical process $M=(M_{t}:t \geq 0)$ for which $M(\phi)=(M_{t}(\phi):t \geq 0)$ is a real-valued local martingale (resp. a martingale). 
In a similar way, a \emph{cylindrical finite variation process} in $\Phi'$ is a cylindrical process $A=(A_{t}:t \geq 0)$ for which $A(\phi)=(A_{t}(\phi):t \geq 0)$ is a real-valued process of finite variation (i.e. if $\Prob$-a.a. of its paths have locally bounded variation).
\end{defi}

If $X$ is a cylindrical semimartingale in $\Phi'$, then each $X(\phi)$ has the decomposition 
\begin{equation} \label{cylinSemimartingaleDecomp}
X_{t}^{\phi}=X_{0}^{\phi}+M_{t}^{\phi}+A_{t}^{\phi}, \quad \forall \, t \geq 0, 
\end{equation}
where $M^{\phi}=(M_{t}^{\phi}:t \geq 0)$ is a real-valued c\`{a}dl\`{a}g local martingale with $M_{0}^{\phi}=0$, and $A^{\phi}=(A_{t}^{\phi}:t \geq 0)$ is a real-valued c\`{a}dl\`{a}g adapted process of finite variation and $A_{0}^{\phi}=0$. If there exists a decomposition \eqref{cylinSemimartingaleDecomp} with each $A^{\phi}$ predictable, we will say that $X$ is a \emph{special cylindrical semimartingale}. It is important to stress the fact that in general the maps $\phi \mapsto M^{\phi}$ and $\phi \mapsto A^{\phi}$ do not define a cylindrical local martingale and a cylindrical finite variation process because they might be not linear operators. For more details see the discussion in Example 3.10 in \cite{ApplebaumRiedle:2010} and Remark 2.2 in \cite{JakubowskiRiedle:2017}.

\begin{defi}\label{defiSemimartingale}
A $\Phi'$-valued process $X=(X_{t}: t \geq 0)$ is a \emph{semimartingale} if the induced cylindrical process is a cylindrical semimartingale. In a completely analogue way we define the concepts of $\Phi'$-valued special semimartingale, martingale, local martingale and process of finite variation. 
\end{defi}

\begin{exam}
One important example of a cylindrical semimartingale in $\Phi'$ is a \emph{cylindrical L\'{e}vy process}, i.e. a cylindrical process $L=(L_{t}:t \geq 0)$ in $\Phi'$ such that  $\forall \, n\in \N$, $\phi_{1}, \dots, \phi_{n} \in \Phi$, the $\R^{n}$-valued process $( (L_{t}(\phi_{1}), \dots, L_{t}(\phi_{n})): t \geq 0)$ is a L\'{e}vy process.

In a similar way, an example of a $\Phi'$-valued semimartingale is a L\'{e}vy processes. Recall that a $\Phi'$-valued process $L=( L_{t} :t\geq 0)$ is called a \emph{L\'{e}vy process} if \begin{inparaenum}[(i)] \item  $L_{0}=0$ a.s., 
\item $L$ has \emph{independent increments}, i.e. for any $n \in \N$, $0 \leq t_{1}< t_{2} < \dots < t_{n} < \infty$ the $\Phi'$-valued random variables $L_{t_{1}},L_{t_{2}}-L_{t_{1}}, \dots, L_{t_{n}}-L_{t_{n-1}}$ are independent,  
\item L has \emph{stationary increments}, i.e. for any $0 \leq s \leq t$, $L_{t}-L_{s}$ and $L_{t-s}$ are identically distributed, and  
\item For every $t \geq 0$ the  probability distribution $\mu_{t}$ of $L_{t}$ is a Radon measure and the mapping $t \mapsto \mu_{t}$ from $\R_{+}$ into the space of Radon probability measures on $\Phi'$ (equipped with the weak topology) is continuous at the origin. \end{inparaenum}

Properties of cylindrical and $\Phi'$-valued L\'{e}vy processes were studied by the author in \cite{FonsecaMora:Levy}. The reader is referred also to  \cite{PerezAbreuRochaArteagaTudor:2005, Ustunel:1984} for other studies on L\'{e}vy and additive processes in the dual of some particular classes of nuclear spaces.
\end{exam}

\begin{exam} Another important example of semimartingales in the dual of a nuclear space are the solutions to certain stochastic evolution equations. 
Let $\Phi$ be a nuclear Fr\'{e}chet space and $M=(M_{t}: t \geq 0)$ be a $\Phi'$-valued square integrable c\`{a}dl\`{a}g martingale. Let $A$ be a continuous linear operator on $\Phi$ that is the infinitesimal generator of a $(C_{0},1)$-semigroup $(S(t): t \geq 0)$ of continuous linear operators on $\Phi$ (see Definition 1.2 in \cite{KallianpurPerezAbreu:1988}). Denote by $A'$ the dual operator of $A$ and by $(S'(t): t \geq 0)$ the dual semigroup to $(S(t): t \geq 0)$. If $\gamma$ is a $\Phi'$-valued square integrable random variable, it is proved in Corollary 2.2  in \cite{KallianpurPerezAbreu:1988} that the homogeneous stochastic evolution equation
$$ dX_{t} = A' X_{t} dt + dM_{t}, \quad X_{0} = \gamma, $$
has a unique solution $X=(X_{t}: t \geq 0)$ given by  
$$ X_{t}=M_{t}+S'(t) \gamma + \int_{0}^{t} S'(t-s) A' M_{s} ds. $$
In particular, $X$ is a $\Phi'$-valued c\`{a}dl\`{a}g semimartingale (see Remark 2.1 in \cite{KallianpurPerezAbreu:1988}).  
\end{exam}

\begin{rema}\label{remaSemiDefiUstunel}
Under the additional assumption that $\Phi'$ is a complete nuclear space, another definition of $\Phi'$-valued semimartingales was introduced by \"{U}st\"{u}nel in \cite{Ustunel:1982} by means of the concept of projective system of stochastic process.  To explain this concept, let $( q_{i}: i \in I)$ be a family of Hilbertian seminorms generating the nuclear topology on $\Phi'$ and let $k_{i}: \Phi' \rightarrow (\Phi')_{q_{i}}$ denotes the canonical inclusion. A \emph{projective system of semimartingales} is a family $X=( X^{i}: i \in I)$ where each $X^{i}$ is a $(\Phi')_{q_{i}}$-valued semimartingale, and  such that if $q_{i} \leq q_{j}$ then $k_{q_{i},q_{j}} X^{j}$ and $X^{i}$ are indistinguishable. A $\Phi'$-valued processes $Y$ is a \emph{semimartingale} if $k_{i} Y= X^{i}$ for each $i \in I$. It is clear that any semimartingale defined in the above sense is also a semimartingale in the sense of Definition \ref{defiSemimartingale}. If $\Phi$ is a Fr\'{e}chet nuclear space the converse is also true as proved by \"{U}st\"{u}nel in \cite{Ustunel:1982-1} (Theorem II.1 and Corollary II.3). Observe that the definition of $\Phi'$-valued semimartingales as projective systems of semimartingales only makes sense if the strong dual $\Phi'$ is complete and nuclear. The above because if $\Phi'$ does not satisfy these assumptions it might not be possible to express $\Phi'$ as a projective limit of Hilbert spaces with Hilbert-Schmidt embeddings. 
\end{rema}

The following is the main result of this section:

\begin{theo}[Regularization of cylindrical semimartingales] \label{theoRegulCylinSemimartigalesNuclear}
Let $X=( X_{t}: t \geq 0) $ be a cylindrical semimartingale in $\Phi'$ such that for each $T>0$, the family of linear maps $( X_{t}: t \in [0,T])$ from $\Phi$ into $L^{0}\ProbSpace$ is equicontinuous (at the origin). Then, there exists a weaker countably Hilbertian topology $\theta$ on $\Phi$ and a $(\widehat{\Phi}_{\theta})'$-valued c\`{a}dl\`{a}g semimartingale $Y= ( Y_{t} : t \geq 0)$, such that for every $\phi \in \Phi$, $\inner{Y}{\phi}= ( \inner{Y_{t}}{\phi}: t \geq 0)$ is a version of $X(\phi)= ( X_{t}(\phi): t \geq 0)$. Moreover, $Y$ is a $\Phi'$-valued, regular, c\`{a}dl\`{a}g semimartingale that is a version of $X$ and it is unique up to indistinguishable versions.  Furthermore, if for each $\phi \in \Phi$ the real-valued semimartingale $X(\phi)$ is continuous, then $Y$ is a  continuous process in $(\widehat{\Phi}_{\theta})'$ and in $\Phi'$. 
\end{theo}
\begin{prf} Since for each $\phi \in \Phi$, the real-valued semimartingale $(X_{t}(\phi): t \geq 0)$ has a c\`{a}dl\`{a}g version (see \cite{DellacherieMeyer}, Section VII.23), the theorem follows using the Regularization Theorem (Theorem 3.2 in \cite{FonsecaMora:2018}). 
\end{prf}

\begin{coro}\label{coroRegulCylSemiUltraborno} If $\Phi$ is an ultrabornological nuclear space, the conclusion of Theorem \ref{theoRegulCylinSemimartigalesNuclear} remains valid if we only assume that each $X_{t}:\Phi \rightarrow L^{0}\ProbSpace$ is continuous (at the origin).   
\end{coro} 
\begin{prf} If $\Phi$ is ultrabornological, the continuity of each $X_{t}$ implies the equicontinuity of the family $( X_{t}: t \in [0,T])$ for each $T>0$ (see \cite{FonsecaMora:2018}, Proposition 3.10). The corollary then follows from Theorem \ref{theoRegulCylinSemimartigalesNuclear}. 
\end{prf}

\begin{rema}\label{remaEquivalEquiconCond} Let $X=(X_{t}: t \geq 0)$ be a cylindrical process in $\Phi'$. The following statements are equivalent (see \cite{VakhaniaTarieladzeChobanyan}, Proposition IV.3.4):
\begin{enumerate}
\item For each $T>0$, the family of linear maps $( X_{t}: t \in [0,T])$ from $\Phi$ into $L^{0}\ProbSpace$ is equicontinuous (at the origin). 
\item For each $T > 0$, the Fourier transforms 
of the family $( X_{t}: t \in [0,T] )$ are equicontinuous (at the origin) in $\Phi$.
\end{enumerate}
Hence the statement of Theorem \ref{theoRegulCylinSemimartigalesNuclear} can be formulated in terms of Fourier transforms. 
\end{rema}

\begin{exam} Consider the space $\mathscr{D}(\R^{d}) \defeq C^{\infty}_{c}(\R^{d})$ of test functions and its dual $\mathscr{D}'(\R^{d})$ the space of distributions.  Let $z=(z_{t}: t \geq 0)$ be a $\R^{d}$-valued semimartingale. Following \cite{Ustunel:1982-1}, we can consider the following cylindrical processes in $\mathscr{D}'(\R^{d})$: for each $t \geq 0$, define the map $X_{t}: \mathscr{D}(\R^{d}) \rightarrow L^{0}\ProbSpace$ by 
$$ X_{t} (\phi)= \delta_{z_{t}}(\phi)=\phi(z_{t}), \quad \forall \, \phi \in \mathscr{D}(\R^{d}), $$
where $\delta_{x}$ denotes the Dirac measure at $x \in \R^{d}$. Each $X_{t}$ is continuous from $\mathscr{D}(\R^{d})$ into $L^{0}\ProbSpace$. Moreover, it is a consequence of It\^{o}'s formula that for each $\phi \in \mathscr{D}(\R^{d})$, $(\phi(z_{t}): t \geq 0) $ is a real-valued semimartingale. Then,  $(X_{t}: t\geq 0)$ satisfies the conditions in Corolary \ref{coroRegulCylSemiUltraborno}, and hence there exists a $\mathscr{D}'(\R^{d})$-valued c\`{a}dl\`{a}g semimartingale $(Y_{t}: t \geq 0)$ with Radon distributions such that $\forall \phi \in \mathscr{D}(\R^{d})$, $\Prob$-a.e. $\inner{Y_{t}}{\phi}=\phi(z_{t})$ for all $t \geq 0$. 
\end{exam}

\begin{exam}\label{examCylLevyRegulari}
Let $X=(X_{t}:t \geq 0)$ is a cylindrical L\'{e}vy process in $\Phi'$ such that or each $T>0$, the family of linear maps $( X_{t}: t \in [0,T])$ from $\Phi$ into $L^{0}\ProbSpace$ is equicontinuous (at the origin). Then its regularized $\Phi'$-valued semimartingale version $Y=(Y_{t}:t \geq 0)$ is indeed a $\Phi'$-valued c\`{a}dl\`{a}g L\'{e}vy process (see Theorem 3.8 in \cite{FonsecaMora:Levy}). Moreover, if $\Phi$ is a  barrelled space, and $L=(L_{t}:t \geq 0)$ is a $\Phi'$-valued L\'{e}vy process, the condition of equicontinuity of the family $( L_{t}: t \in [0,T])$ from $\Phi$ into $L^{0}\ProbSpace$ for all $T>0$ is always satisfied  (see Corollary 3.11 in \cite{FonsecaMora:Levy}). 
\end{exam}


If $X$ is a $\Phi'$-valued semimartingale satisfying the  equicontinuity condition in the statement of Theorem \ref{theoRegulCylinSemimartigalesNuclear}, then $X$ has a  c\`{a}dl\`{a}g semimartingale version. However, as the next result shows there is a large class of examples of nuclear spaces where every $\Phi'$-valued semimartingale always have a  c\`{a}dl\`{a}g semimartingale version. 

\begin{prop}\label{propCadlagVerSemimUltrabNuclear}
The conclusion of Theorem \ref{theoRegulCylinSemimartigalesNuclear} remains valid if $\Phi$ is an ultrabornological nuclear space and $X=(X_{t}: t \geq 0)$ is a $\Phi'$-valued semimartingale such that the probability distribution of each $X_{t}$ is a Radon measure.
\end{prop}
\begin{prf}
If the probability distribution of $X_{t}$ is a Radon measure, then by Theorem 2.10 in \cite{FonsecaMora:2018} and the fact that $\Phi$ being ultrabornological is also barrelled (see \cite{NariciBeckenstein}, Theorem 11.12.2) imply that each $X_{t}: \Phi \rightarrow L^{0} \ProbSpace$ is continuous. The result now follows from Corollary \ref{coroRegulCylSemiUltraborno}.  
\end{prf}

\begin{coro} \label{coroCadlagSemimarFrecNucleSpace}
If $\Phi$ is a Fr\'{e}chet nuclear space or the countable inductive limit of Fr\'{e}chet nuclear spaces, then each 
$\Phi'$-valued semimartingale $X=(X_{t}: t \geq 0)$ possesses a c\`{a}dl\`{a}g semimartingale version (unique up to indistinguishable versions).
\end{coro}
\begin{prf}
If $\Phi$ is a Fr\'{e}chet nuclear space or a countable inductive limit of Fr\'{e}chet nuclear spaces, then every Borel measure on $\Phi'_{\beta}$ is a Radon measure (see Corollary 1.3 of Dalecky and Fomin \cite{DaleckyFomin}, p.11). In particular, for each $t \geq 0$ the probability distribution of $X_{t}$ is Radon. The result now follows from Proposition \ref{propCadlagVerSemimUltrabNuclear}. 
\end{prf}

Let $X=( X_{t}: t \geq 0)$ be a cylindrical semimartingale in $\Phi'$. Clearly $X$ induces a linear map $\phi \mapsto X(\phi)$ from $\Phi$ into $S^{0}$. The next result shows that the continuity of this map is equivalent to the equicontinuity condition in the statement of Theorem \ref{theoRegulCylinSemimartigalesNuclear}. 

\begin{prop} \label{propCylSemimContOpeSpaceSemim}
Let $X=( X_{t}: t \geq 0) $ be a cylindrical semimartingale in $\Phi'$. The following statements are equivalent:
\begin{enumerate} 
\item The linear mapping $X:\Phi \rightarrow S^{0}$, $\phi \mapsto X(\phi)$, is continuous. 
\item  For each $T>0$, the family of linear maps $( X_{t}: t \in [0,T])$ from $\Phi$ into $L^{0}\ProbSpace$ is equicontinuous (at the origin).
\end{enumerate}
If any of the above is satisfied, there exists exists a weaker countably Hilbertian topology $\theta$ on $\Phi$ such that $X$ extends to a continuous map from $\widehat{\Phi}_{\theta}$ into $S^{0}$. \end{prop}
\begin{prf} We first prove $(1) \Rightarrow (2)$. Observe that because convergence in $\mathcal{S}^{0}$ under the metric $d$ implies uniform convergence in probability on compact subsets in $\R_{+}$ (see \cite{Memin:1980}, Remarque II.2), then the continuity of the mapping $X:\Phi \rightarrow S^{0}$ implies that for each $T>0$, the family of mappings $( X_{t}: t \in [0,T])$ from $\Phi$ into $L^{0} \ProbSpace$ is equicontinuous (see the proof of Lemma 3.7 in \cite{FonsecaMora:2018}).

Now we will prove $(2) \Rightarrow (1)$ and the last implication in the statement of Proposition \ref{propCylSemimContOpeSpaceSemim}. Let $\theta$ be a weaker countably Hilbertian topology on $\Phi$ and a $(\widehat{\Phi}_{\theta})'$-valued c\`{a}dl\`{a}g process $Y= ( Y_{t} : t \geq 0)$ as in the conclusion of Theorem \ref{theoRegulCylinSemimartigalesNuclear}. Because for each $\phi \in \Phi$, $\inner{Y}{\phi}= ( \inner{Y_{t}}{\phi} : t \geq 0)$ is a version of $X(\phi)= ( X_{t}(\phi) : t \geq 0) \in S^{0}$, and because $\Phi$ is dense in $\widehat{\Phi}_{\theta}$, then $Y$ defines a linear map from $\widehat{\Phi}_{\theta}$ into $S^{0}$. Furthermore, the continuity for each $t\geq 0$ of $Y_{t}$ on $\widehat{\Phi}_{\theta}$ easily implies that $Y$ is a closed linear map from $\widehat{\Phi}_{\theta}$ into $S^{0}$. But because $\widehat{\Phi}_{\theta}$ is ultrabornological and $S^{0}$ is a complete, metrizable, topological vector space, the closed graph theorem (see \cite{NariciBeckenstein}, Theorem 14.7.3, p.475) implies that the map defined by $Y$ is continuous from $\widehat{\Phi}_{\theta}$ into $S^{0}$. However, because for each $\phi \in \Phi$ we have $\inner{Y}{\phi}=X(\phi)$ in $S^{0}$, then $X$ extends to a continuous map from $\widehat{\Phi}_{\theta}$ into $S^{0}$. Furthermore, because the inclusion from $\Phi$ into $\widehat{\Phi}_{\theta}$ is linear and continuous, then $X:\Phi \rightarrow S^{0}$ is continuous.      
\end{prf}

If in the proof of Proposition \ref{propCylSemimContOpeSpaceSemim} we use Corollary \ref{coroRegulCylSemiUltraborno} instead of Theorem \ref{theoRegulCylinSemimartigalesNuclear} we obtain the following conclusion: 

\begin{prop}\label{coroCylSemimContOpeSpaceSemimUltra} If $\Phi$ is an ultrabornological nuclear space, we can replace (2) in Proposition \ref{propCylSemimContOpeSpaceSemim} by the assumption that each $X_{t}:\Phi \rightarrow L^{0}\ProbSpace$ is continuous (at the origin).
\end{prop}

\begin{rema}\label{remaContMapCylSemiLCS} The implication $(1) \Rightarrow (2)$ in Proposition \ref{propCylSemimContOpeSpaceSemim} remains true if $\Phi$ is only a locally convex space (see the proof of Lemma 3.7 in \cite{FonsecaMora:2018}).
\end{rema}

\subsection{Regularization of Some Classes of Cylindrical Semimartingales}

In this section we study how the results obtained in the above section specialize when we restrict our attention to cylindrical local martingales and cylindrical processes of finite variation. We start with the following regularization result that follows easily from Theorem \ref{theoRegulCylinSemimartigalesNuclear}. 

\begin{prop}\label{propRegulaLocalMarBoundVaria} If in 
Theorem \ref{theoRegulCylinSemimartigalesNuclear} or in Corollary \ref{coroRegulCylSemiUltraborno} the cylindrical process $X$ is a cylindrical local martingale (resp. a cylindrical process of finite variation), then the regularized version $Y$ of $X$ is a local martingale (resp. a process of finite variation).  
\end{prop}

\begin{rema}
If $X=( X_{t}: t \geq 0) $ is a cylindrical local martingale (resp. a cylindrical process of finite variation) satisfying condition (2) in Proposition \ref{propCylSemimContOpeSpaceSemim}, then it is not true in general that $X$ defines a continuous operator from $\Phi$ into the space of real-valued local martingales $\mathcal{M}_{loc}$ (resp. of real-valued processes of finite variation $\mathcal{V}$). This is a consequence of the fact that $\mathcal{M}_{loc}$ and $\mathcal{V}$ are not closed subspaces of $S^{0}$ (see \cite{Emery:1979}). 
\end{rema}

Regularization of cylindrical martingales were studied by the author in \cite{FonsecaMora:2018}. Observe that if the cylindrical martingale has $n$-moments, for $n \geq 2$, we obtain a regularized version taking values in some Hilbert space $\Phi'_{q}$. 

\begin{theo}[\cite{FonsecaMora:2018}, Theorem 5.2] \label{theoRegMartingales} 
Let $X=( X_{t}:t\geq 0)$ be a cylindrical martingale in $\Phi'$ such that for each $t \geq 0$ the map $X_{t}: \Phi \rightarrow L^{0} \ProbSpace$ is continuous. Then, $X$ has a $\Phi'$-valued c\`{a}dl\`{a}g version $Y=( Y_{t} : t\geq 0)$. Moreover, we have the following:
\begin{enumerate}
\item If $X$ is $n$-th integrable with $n \geq 2$, then for each $T>0$ there exists a continuous Hilbertian semi-norm $q_{T}$ on $\Phi$ such that $( Y_{t} : t \in [0, T])$ is a $\Phi'_{q_{T}}$-valued c\`{a}dl\`{a}g martingale satisfying $\Exp \left( \sup_{t \in [0, T] } q'_{T} ( Y_{t} )^{n} \right) < \infty$.  
\item Moreover, if for $n \geq 2$, $\sup_{t \geq 0} \Exp \left( \abs{X_{t}(\phi)}^{n} \right)< \infty$ for each $\phi \in \Phi$, then there exists a continuous Hilbertian semi-norm $q$ on $\Phi$ such that $Y$ is a $\Phi'_{q}$-valued c\`{a}dl\`{a}g martingale satisfying $\Exp \left( \sup_{t \geq 0}  q'( Y_{t} )^{n} \right) < \infty$.    
\end{enumerate}
If for each $ \phi \in \Phi$ the real-valued process $( X_{t}(\phi) : t \geq 0)$ has a continuous version, then $Y$ can be chosen to be continuous and such that it satisfies \emph{(1)}$-$\emph{(2)} above replacing the property c\`{a}dl\`{a}g by continuous. 
\end{theo}

\begin{exam}
The following simple example illustrates the usefulness of Theorem \ref{theoRegMartingales}. Let $B=(B_{t}: t \geq 0)$ denotes a real-valued Brownian motion. For every $t \geq 0$ define 
$$X_{t}(\phi)= \int_{0}^{t} \phi(s) dB_{s}, \quad \forall \phi \in \mathcal{S}(\R). $$
From the properties of the It\^{o} stochastic integral we have that $X=(X_{t}: t \geq 0)$ is a cylindrical square integrable continuous martingale in the space of tempered distributions $\mathcal{S}'(\R)$. Moreover, from It\^{o} isometry we have $ \Exp \abs{X_{t}(\phi)}^{2} = \norm{\mathbbm{1}_{[0,t]} \phi}^{2}_{L^{2}(\R)}$ $\forall \phi \in \mathcal{S}(\R)$ 
and since the canonical inclusion from $\mathcal{S}(\R)$ into $L^{2}(\R)$ is continuous, it follows that each $X_{t}: \mathcal{S}(\R) \rightarrow L^{0}\ProbSpace$ is also continuous. Theorefore, Theorem \ref{theoRegMartingales} shows that $X$ has a version $Y=(Y_{t}: t \geq 0)$ that is a square integrable continuous martingale in the space of tempered distributions $\mathcal{S}'(\R)$. 

We can learn more about $X$ if we analyse further some families of norms on $\mathcal{S}(\R)$. For every $p \in \R$, define on $\mathcal{S}(\R)$ the norm:
$$ \norm{\phi}^{2}_{p}=\sum_{n=1}^{\infty} \left( n+\frac{1}{2} \right)^{2p} \inner{\phi}{\phi_{n}}^{2}_{L^{2}(\R)}, $$ 
where the sequence of Hermite functions $(\phi_{n}: n=1,2, \dots)$ is defined as:
$$ \phi_{n+1}(x)=\sqrt{g(x)} \, h_{n}(x), \quad n=0,1,2, \dots, $$
for $\displaystyle{g(x)=(\sqrt{2\pi})^{-1} \exp(-x^{2}/2)}$ and where $(h_{n}: n=0,1,2, \dots)$ is the sequence of Hermite polynomials:
$$h_{n}(x)=\frac{(-1)^{n}}{\sqrt{n!}} g(x)^{-1} \frac{d^{n}}{dx^{n}} g(x), \quad n=0,1,2, \dots. $$
It is known (see e.g. Theorem 1.3.2 in  \cite{KallianpurXiong}) that the topology in $\mathcal{S}(\R)$ is generated by the increasing sequence of Hilbertian norms $(\norm{\cdot}_{p}: p=0,1,\dots)$. Moreover, if $\mathcal{S}_{p}$ denotes the completion of $S(\R)$ when equipped with the norm $\norm{\cdot}_{p}$, then it follows that $\mathcal{S}_{p}'=\mathcal{S}_{-p}$. 

Then since $\sup_{t \geq 0} \Exp \abs{X_{t}(\phi)}^{2} =  \norm{ \phi}^{2}_{L^{2}(\R)}< \infty$ $\forall \phi \in \mathcal{S}(\R)$, Theorem \ref{theoRegMartingales}(2) shows that there exists some $p \in \N$ such that $Y$ is a continuous square integrable martingale in $\mathcal{S}_{-p}$ satisfying $\Exp \left( \sup_{t \geq 0}   \norm{Y_{t}}_{-p}^{2}  \right) < \infty$. 
\end{exam}


In the next result we study regularization for cylindrical processes of locally integrable variation. Observe that in this case one can obtain a regularized version with paths that on a bounded time interval are of bounded variation in some Hilbert space $\Phi'_{\rho}$. 
We denote by $\mathcal{A}$ the Banach space of real-valued processes of integrable variation $a=(a_{t}: t \geq 0)$ equipped with the norm of expected total variation $\norm{a}_{\mathcal{A}}= \Exp \int^{\infty}_{0} \abs{d a_{t}}$. Similarly, $\mathcal{A}_{loc}$ denotes the linear space of real-valued predictable processes of finite variation with locally integrable variation, equipped with the topology of uniform convergence in probability in the total variation on every compact interval $[0,T]$ of $\R_{+}$. 

\begin{theo} \label{theoRegulCylSemIntegBouVari}
Let $\tilde{A}=(\tilde{A}_{t}: t \geq 0)$ be a cylindrical process of locally integrable variation, i.e. such that $\tilde{A}(\phi) \in \mathcal{A}_{loc}$ for each $\phi \in \Phi$. Assume further that for each $T>0$, the family of linear maps $( \tilde{A}_{t}: t \in [0,T])$ from $\Phi$ into $L^{0}\ProbSpace$ is equicontinuous (at the origin). Then, the cylindrical process $\tilde{A}$ has a $\Phi'$-valued regular c\`{a}dl\`{a}g version $A=(A_{t}: t \geq 0)$ (unique up to indistinguishable versions) satisfying that for every $\omega \in \Omega$ and $T>0$, there exists a continuous Hilbertian semi-norm $\varrho=\varrho(\omega,T)$ on $\Phi$ such that the map $t \mapsto A_{t}(\omega)$ defined on $[0,T]$ has bounded variation  in $\Phi'_{\varrho}$. 
\end{theo}
\begin{prf} 
First, from Proposition \ref{propRegulaLocalMarBoundVaria} there exists a weaker countably Hilbertian topology $\theta$ on $\Phi$, and a $(\widehat{\Phi}_{\theta})'$-valued  adapted c\`{a}dl\`{a}g process $A=(A_{t}: t \geq 0)$ such that $\Prob$-a.e. $\inner{A_{t}}{\phi}=\tilde{A}_{t}(\phi)$ $\forall t \geq 0$, $\phi \in \Phi$. From the above equality it follows that for each $\phi \in \Phi$, $\inner{A}{\phi} \in \mathcal{A}_{loc}$. Therefore, from Proposition \ref{propCylSemimContOpeSpaceSemim} and since $\mathcal{A}_{loc}$ is a closed subspace of $S^{0}$ (see \cite{Memin:1980}, Th\'{e}or\`{e}me IV.7), the process $A$ defines a linear and continuous map from $\widehat{\Phi}_{\theta}$ into $\mathcal{A}_{loc}$. We will use the above to show that the paths of $A$ satisfy the bounded variation properties indicated in the statement of the theorem.  We will benefit from ideas taken from \cite{BadrikianUstunel:1996}. 

Fix $T>0$. For every $\epsilon >0$, using the linearity and continuity of the map $A$ from $\widehat{\Phi}_{\theta}$ into $\mathcal{A}_{loc}$, and by following similar arguments to those used in the proof of Lemma 3.3 in \cite{BadrikianUstunel:1996} there exists a $\theta$-continuous Hilbertian seminorm $p$ on $\Phi$, such that 
\begin{equation} \label{equiconContMapBoundVariation}
\sup_{\Delta} \Exp \left[ \sup_{\norm{y}_{l_{\Delta}^{\infty}} \leq 1} \abs{1-e^{i\inner{A^{\Delta}(\phi)}{y}}} \right] \leq \epsilon + 2 p(\phi)^{2}, \quad \forall \, \phi \in \Phi, 
\end{equation} 
where the $\sup$ is taken with respect to an increasing sequence of finite partitions $\Delta$ of $[0,T]$ in such a way that the supremum can be attained as a monotone limit. Let us explain the notation used in \eqref{equiconContMapBoundVariation}. 
If $\Delta=\{ 0=t_{0} < t_{1}< \dots < t_{n}=T\}$ is a finite partition of $[0,T]$, $A^{\Delta}$ denotes the linear and continuous map:
$$ \phi \mapsto A^{\Delta}(\phi) \defeq (A(\phi)(t_{1})-A(\phi)(t_{0}), \dots, A(\phi)(t_{n})-A(\phi)(t_{n-1})), $$
from $\widehat{\Phi}_{\theta}$ into $L^{0}(l^{1}_{\Delta})$, where $l^{1}_{\Delta}$ denotes the space $\R^{n}$ equipped with the $l^{1}$-norm $\norm{y}_{l_{\Delta}^{1}}=\sum_{k=1}^{n} \abs{y_{k}}$ for $y=(y_{1}, \dots, y_{n}) \in \R^{n}$. Recall that in $\R^{n}$ the dual norm to the $l^{\infty}$-norm $\norm{y}_{l_{\Delta}^{\infty}}=\max_{1 \leq k \leq n} \abs{y_{k}}$ for $y=(y_{1}, \dots, y_{n}) \in \R^{n}$ is the $l^{1}$-norm, i.e. we have that  $\norm{x}_{l_{\Delta}^{1}}=\sup \{\abs{\inner{x}{y}}: \norm{y}_{l_{\Delta}^{\infty}} \leq 1 \}$, where $\inner{\cdot}{\cdot}$ denotes the scalar product in $\R^{n}$ and $l^{\infty}_{\Delta}$ denotes the space $\R^{n}$ equipped with the $l^{\infty}$-norm.  

Let $( p_{m}: m \in \N)$ be an increasing sequence of continuous Hilbertian seminorms on $\Phi$ generating the topology $\theta$. Since $\Phi$ is a nuclear space, we can find and increasing sequence of continuous Hilbertian seminorms $( q_{m}: m \in \N)$ on $\Phi$ such that $\forall m \in \N$, $p_{m} \leq q_{m}$, and the inclusion $i_{p_{m},q_{m}}$ is Hilbert-Schmidt. We denote by $\alpha$ the countably Hilbertian topology on $\Phi$ generated by the seminorms $( q_{m}: m \in \N)$. By construction, the topology $\alpha$ is finer than $\theta$. 
 
Let $( \epsilon_{n}: n \in \N)$ be a sequence of positive real numbers such that $\sum_{n} \epsilon_{n}< \infty$. Then, there exists a subsequence $( p_{m_{n}}: n \in \N)$ of $( p_{m}: m \in \N)$ such that for each $n \in \N$, $\epsilon_{n}$ and $p_{m_{n}}$ satisfy \eqref{equiconContMapBoundVariation}. 
To keep the notation simple, we will denote $p_{m_{n}}$ by $p_{n}$ and the corresponding $q_{m_{n}}$ by $q_{n}$.

Let $C>0$ and let $\Delta$ be any member of the increasing sequence of finite partitions of $[0,T]$ for which the supremum in \eqref{equiconContMapBoundVariation} is attained. Let $(\phi^{q_{n}}_{j}: j \in \N) \subseteq \Phi$ be a complete orthonormal system in $\Phi_{q_{n}}$. Then, similar arguments to those used in the proof of Lemma 3.8 in \cite{FonsecaMora:2018} shows that 

\begin{flalign*}
& \Prob \left( \sup_{q_{n}(\phi) \leq 1} \norm{ A^{\Delta}(\phi) }_{l^{1}_{\Delta}} > C \right) & \\
& = \Prob \left( \sup_{q_{n}(\phi) \leq 1} \sup_{\norm{y}_{l^{\infty}_{\Delta}} \leq 1} \abs{\inner{A^{\Delta}(\phi)}{y}} > C \right) & \\
& \leq \frac{\sqrt{e}}{\sqrt{e}-1} \Exp \left[ 1- \exp \left\{ -\frac{1}{2C^{2}} \sup_{q_{n}(\phi) \leq 1} \sup_{\norm{y}_{l^{\infty}_{\Delta}} \leq 1} \inner{A^{\Delta}(\phi)}{y}^{2} \right\}  \right] & \\
& = \lim_{m \rightarrow \infty} \frac{\sqrt{e}}{\sqrt{e}-1} \Exp  \left[ \sup_{\norm{y}_{l^{\infty}_{\Delta}} \leq 1}  \left( 1- \exp \left\{ -\frac{1}{2C^{2}} \sum_{j=1}^{m}  \inner{A^{\Delta}(\phi^{q_{n}}_{j})}{y}^{2} \right\} \right) \right]  & \\
& \leq \lim_{m \rightarrow \infty} \frac{\sqrt{e}}{\sqrt{e}-1} 
\int_{\R^{m}} \Exp  \left[ \sup_{\norm{y}_{l^{\infty}_{\Delta}} \leq 1} \,  \abs{ 1- \exp \left\{ i \sum_{j=1}^{m} z_{j}  \inner{A^{\Delta}(\phi^{q_{n}}_{j})}{y} \right\} } \right]  \otimes_{j=1}^{m} N_{C} (dz_{j}),  
\end{flalign*}
where $N_{C}$ denotes the centred Gaussian measure on $\R$ with variance $1/C^{2}$. Now, from \eqref{equiconContMapBoundVariation} the last term can be majorated by
\begin{flalign*}
& \lim_{m \rightarrow \infty} \frac{\sqrt{e}}{\sqrt{e}-1} 
\int_{\R^{m}}  \left[ \epsilon_{n} + 2p_{n} \left( \sum_{j=1}^{m} z_{j}\phi^{q_{n}}_{j} \right)^{2}  \right]  \otimes_{j=1}^{m} N_{C} (dz_{j})   & \\   
& \leq  \lim_{m \rightarrow \infty} \frac{\sqrt{e}}{\sqrt{e}-1} \left[ \epsilon_{n} + \frac{2}{C^{2}}  \sum_{j=1}^{m}p_{n}\left(\phi^{q_{n}}_{j} \right)^{2}  \right]  & \\
& = \frac{\sqrt{e}}{\sqrt{e}-1}   \left[ \epsilon_{n} + \frac{2}{C^{2}}  \norm{ i_{p_{n}, q_{n}} }^{2}_{\mathcal{L}_{2}(\Phi_{q_{n}}, \Phi_{p_{n}}) } \right].   
\end{flalign*}  
From the above bound we have
\begin{eqnarray*}
\Prob \left( \sup_{q_{n}(\phi) \leq 1} \int^{T}_{0} \abs{d A_{t}(\phi)} >C \right) 
& = & \Prob \left( \sup_{q_{n}(\phi) \leq 1} \sup_{\Delta} \norm{ A^{\Delta}(\phi) }_{l^{1}_{\Delta}} > C \right) \\
& = & \sup_{\Delta} \Prob \left( \sup_{q_{n}(\phi) \leq 1} \norm{ A^{\Delta}(\phi) }_{l^{1}_{\Delta}} > C \right) \\
& \leq & \frac{\sqrt{e}}{\sqrt{e}-1}   \left[ \epsilon_{n} + \frac{2}{C^{2}}  \norm{ i_{p_{n}, q_{n}} }^{2}_{\mathcal{L}_{2}(\Phi_{q_{n}}, \Phi_{p_{n}}) } \right],   
\end{eqnarray*}
where, as before, the $\sup$ is taken with respect to an increasing sequence of finite partitions of $[0,T]$ for which the supremum in \eqref{equiconContMapBoundVariation} is attained. 

Then, by taking limit as $C \rightarrow \infty$ we get that 
\begin{equation} \label{lowBoundPartProbBoundVari}
\Prob \left( \sup_{q_{n}(\phi) \leq 1} \int^{T}_{0} \abs{d A_{t}(\phi)} < \infty \right)  
\geq 1 - \frac{\sqrt{e}}{\sqrt{e}-1} \epsilon_{n}.  
\end{equation}
Now, if for every $ n \in \N$ we take 
$$\Lambda_{n} = \left\{ \sup_{q_{n}(\phi) \leq 1} \int^{T}_{0} \abs{d A_{t}(\phi)} < \infty \right\} \quad \mbox{and} \quad  \Omega_{T} = \bigcup_{N \in \N} \bigcap_{n \geq N} \Lambda_{n},$$
it follows from  \eqref{lowBoundPartProbBoundVari}, our assumption $\sum_{n} \epsilon_{n} < \infty$, and the Borel-Cantelli lemma that $\Prob (\Omega_{T})=1$.  

Now, recall  that the process $A$ was obtained from regularization of the cylindrical process $\tilde{A}$. It is a consequence of the construction of this regularized version (see Remark 3.9 in \cite{FonsecaMora:2018}) that there exists $\Gamma\subseteq \Omega$ with $\Prob (\Gamma)=1$ such that for all $\omega \in \Gamma$, for every $t > 0$ there exists a $\theta$-continuous Hilbertian semi-norm $p=p(\omega,t)$ on $\Phi$ such that the map $s \mapsto A_{s}(\omega)$ is c\`{a}dl\`{a}g from $[0,t]$ into the Hilbert space $\Phi'_{p}$. 

If $\omega \in \Gamma \cap \Omega_{T}$, it follows from the construction of the set $\Omega_{T}$ that  there exists an $\alpha$-continuous Hilbertian seminorm $q=q(\omega,T)$ on $\Phi$, with $p \leq q$,  and such that $\sup_{q(\phi) \leq 1} \int^{T}_{0} \abs{d A_{t} (\phi)(\omega)} < \infty $. The above clearly implies that $\phi \mapsto \inner{i'_{p,q} A (\omega)}{\phi}$ is a linear and continuous mapping from $\Phi_{q}$ into the Banach space $BV([0,T])$ of real-valued functions that are of bounded variation on $[0,T]$ equipped with the total variation norm. 

Let $\varrho$ be a continuous Hilbertian semi-norm on $\Phi$ such that $q \leq \varrho$ and $i_{q,\varrho}$ is Hilbert-Schmidt. Then, the dual operator $i'_{q, \varrho}$ is Hilbert-Schmidt and hence is $1$-summing (see \cite{DiestelJarchowTonge},  Corollary 4.13, p.85). Then, from the Pietsch domination theorem (see \cite{DiestelJarchowTonge}, Theorem 2.12, p.44) there exists a constant $C>0$, and a Radon probability measure $\nu$ on the unit ball $B^{*}_{q}(1)$ of $\Phi_{q}$ (equipped with the weak topology) such that, 
\begin{equation}\label{integralInequalitySummingOperator}
\varrho'(i'_{q,\varrho} f ) \leq C  \int_{B^{*}_{q}(1)} \abs{\inner{f}{\phi}} \nu(d\phi), \quad \forall \, f \in \Phi'_{q}.    
\end{equation}
Then, the continuity of the map $\phi \mapsto \inner{A (\omega) \circ i_{p,q}}{\phi}$ and  \eqref{integralInequalitySummingOperator} implies that 
\begin{eqnarray*}
\sum_{\Delta} \varrho' \left( i'_{p,\varrho} A_{t_{i+1}}(\omega) -i'_{p,\varrho} A_{t_{i}}(\omega) \right) 
& = & \sum_{\Delta} \varrho' \left( i'_{q,\varrho} \left( i'_{p,q} A_{t_{i+1}}(\omega) -i'_{p,q} A_{t_{i}}(\omega) \right) \right) \\
& \leq  & C  \int_{B^{*}_{q}(1)} \sum_{\Delta} \abs{\inner{i'_{p,q} A_{t_{i+1}}(\omega) -i'_{p,q} A_{t_{i}}(\omega)}{\phi}} \nu(d\phi) \\
& \leq & \norm{ i_{p, q} }_{\mathcal{L}_{2}(\Phi_{q}, \Phi_{p}) } \sup_{q(\phi) \leq 1} \int^{T}_{0} \abs{d A_{t}(\phi)} < \infty. 
\end{eqnarray*}
The above bound is uniform on $\Delta$, therefore
$$  \sup_{\Delta} \sum_{\Delta} \varrho' \left( i'_{p,\varrho} A_{t_{i+1}}(\omega) -i'_{p,\varrho} A_{t_{i}}(\omega) \right) < \infty, $$
and hence $A$ is a $\Phi'$-valued version of $\tilde{A}$ whose paths satisfy the properties on the statement of the theorem. 
\end{prf}

We finalize this section by studying regularization of $S^{p}$-cylindrical semimartingales. The next results is a generalization of Theorem III.1 in  \cite{Ustunel:1982}. 

\begin{theo} \label{theoRegulSpCylinSemim}
Let $X=(X_{t}: t \geq 0)$ be a $S^{p}$-cylindrical semimartingale $(1 \leq p < \infty)$, i.e. such that $X(\phi) \in S^{p}$ for each $\phi \in \Phi$. Assume further that for each $T>0$, the family of linear maps $( X_{t}: t \in [0,T])$ from $\Phi$ into $L^{0}\ProbSpace$ is equicontinuous (at the origin). Then, there exists a continuous Hilbertian seminorm $q$ on $\Phi$ such that  $X$ has a c\`{a}dl\`{a}g version $Y=(Y_{t}: t \geq 0)$ that is a $S^{p}$-semimartingale in $\Phi'_{q}$. Moreover, $Y$ is unique up to indistinguishable versions as a $\Phi'$-valued process.  
\end{theo}
\begin{prf} 
First, a closed graph theorem argument similar to the one used in the proof of Proposition \ref{propCylSemimContOpeSpaceSemim} shows that $X$ defines a continuous and linear operator from $\Phi$ into $S^{p}$. Since $\Phi$ is a nuclear space and $S^{p}$ is Banach, the map $X: \Phi \rightarrow S^{p}$ is nuclear and then it has a representation (see \cite{Schaefer}, Theorems III.7.1-2): 
$$ X=\sum_{i=1}^{\infty} \lambda_{i} F_{i} \otimes x^{i} , $$
where $(\lambda_{i}) \in l^{1}$, $(F_{i}) \subseteq \Phi'$ equicontinuous, and $(x^{i}) \subseteq S^{p}$ bounded. Choose any $\epsilon>0$, and for each $i$ a local martingale $m^{i}$ and a process of integrable variation $a^{i}$ such that $x^{i}=m^{i}+a^{i}$ and 
$$  \norm{ [m^{i},m^{i}]_{\infty}^{1/2} + \int_{0}^{\infty} \abs{d a^{i}_{s} } }_{L^{p}\ProbSpace} <  \norm{x^{i}}_{S^{p}} + \epsilon. $$
Now, since $(F_{i})$ is equicontinuous and $\Phi$ is nuclear, there exists a continuous Hilbertian seminorm $q$ on $\Phi$ such that $(F_{i}) \subseteq B_{q}(1)^{0}$, where $B_{q}(1)^{0}$ denotes the polar set of the unit ball $B_{q}(1)$ of $q$. Define 
$$ M_{t}(\omega)= \sum_{i=1}^{\infty} \lambda_{i} F_{i}  \, m_{t}^{i}(\omega), \quad \forall \, t \geq 0, \, \omega \in \Omega,$$
and 
$$ A_{t}(\omega)= \sum_{i=1}^{\infty} \lambda_{i} F_{i} \,  a_{t}^{i}(\omega), \quad \forall \, t \geq 0, \, \omega \in \Omega.$$
Then, following similar arguments to those used in the proof of Theorem III.1 in  \cite{Ustunel:1982}) shows that $M=(M_{t}: t\geq 0)$ defines a $\Phi'_{q}$-valued c\`{a}dl\`{a}g local martingale with $\Exp ( \sup_{t \geq 0} q'(M_{t})^{p}) < \infty$, and that $A=(A_{t}: t\geq 0)$ defines a $\Phi'_{q}$-valued c\`{a}dl\`{a}g process of $p$-integrable variation. Hence, if we define $Y=(Y_{t}: t \geq 0)$ by $Y_{t}=M_{t}+A_{t}$ $\forall t \geq 0$, then $Y$ is is a $\Phi'_{q}$-valued c\`{a}dl\`{a}g  $S^{p}$-semimartingale. Moreover, it is clear from the definition of $Y$ that it is a version of $X$. 
\end{prf}

\begin{coro}
If $\Phi$ is an ultrabornological nuclear space, the result in Theorem \ref{theoRegulSpCylinSemim} is valid if we only assume that each $X_{t}:\Phi \rightarrow L^{0}\ProbSpace$ is continuous (at the origin).
\end{coro}

\section{Canonical Representation of Semimartingales in Duals of Nuclear Spaces} \label{secDecomSemi}

\begin{assu}
Throught this section and unless otherwise specified $\Phi$ will always denote a nuclear space. 
\end{assu}

The aim of this section is to prove the following theorem that provides a detailed canonical representation for $\Phi'$-valued semimartingales satisfying the equicontinuity condition in the statement of Theorem \ref{theoRegulCylinSemimartigalesNuclear}. We will show later (see Proposition \ref{propSemimaDecompUltrabNucleSpace}) that this equicontinuity condition can be discarded if the nuclear space is ultrabornological, therefore generalizing the canonical representation for semimartingales in the dual Fr\'{e}chet nuclear space obtained by P\'{e}rez-Abreu in \cite{PerezAbreu:1988} (see Corollary  \ref{coroDecomSemimarFrecNucleSpace} below).  

\begin{theo}[Semimartingale canonical representation] \label{theoSemirtDecompNucleSpace}
Let $X=(X_{t}:t \geq 0)$ be $\Phi'$-valued, adapted, c\`{a}dl\`{a}g  semimartingale such that for each $T>0$,  the family of linear maps $( X_{t}: t \in [0,T])$ from $\Phi$ into $L^{0}\ProbSpace$ is equicontinuous (at the origin). Given a continuous Hilbertian seminorm $\rho$ on $\Phi$, for each $t \geq 0$, $X_{t}$ admits the unique representation  
\begin{equation}\label{eqSemimartDecompNucleSpace}
X_{t}=X_{0}+ A_{t} +M^{c}_{t}+ \int_{0}^{t} \int_{B_{\rho'}(1)} f d(\mu-\nu)(s,f) + \int_{0}^{t} \int_{B_{\rho'}(1)^{c}} f d\mu (s,f),
\end{equation}
where
\begin{enumerate}
\item $X_{0}$ is a $\mathcal{F}_{0}$-measurable $\Phi'$-valued regular random variable, 
\item $A=(A_{t}: t \geq 0)$ is a $\Phi'$-valued regular predictable process with uniformly bounded jumps satisfying that for every $\omega \in \Omega$ and $T>0$, there exists a continuous Hilbertian semi-norm $\varrho=\varrho(\omega,T)$ on $\Phi$ such that the map $t \mapsto A_{t}(\omega)$ defined on $[0,T]$ has bounded variation  in $\Phi'_{\varrho}$,
\item $M^{c}=(M^{c}_{t}: t \geq 0)$ is a $\Phi'$-valued regular continuous local martingale with $M^{c}_{0}=0$, 
\item $\mu(\omega; (0,t]; \Gamma)= \sum_{0<s \leq t} \mathbbm{1}_{\{ \Delta X_{s} \in \Gamma \}}$, $\Gamma \in \mathcal{B}(\Phi'_{0})$ with $0 \notin \overline{\Gamma}$ (here $\Phi'_{0} \defeq \Phi' \setminus \{0 \}$), is the integer-valued random measure of the jumps of $X$ with (predictable) compensator measure $\nu=\nu(\omega, dt,df)$ that satisfies the conditions: 
\begin{enumerate}
\item $\nu(\omega; \{ 0\}; \Phi')= \nu(\omega; \R_{+}; \{0\})=0$, 
\item $\nu(\omega; \{ t\}; \Phi') \leq 1$ $\forall t > 0$, 
\item $\int^{t}_{0} \int_{\Phi'} \, \abs{\inner{f}{\phi}}^{2} \wedge 1 \,  \nu (ds, df)< \infty$, $\forall \phi \in \Phi$, $t>0$,
\end{enumerate}
\item $\left( \int_{0}^{t} \int_{B_{\rho'}(1)} f d(\mu-\nu)(s,f): t \geq 0 \right)$, is a $\Phi'$-valued regular purely discontinuous local martingale with uniformly bounded jumps satisfying 
$$ \inner{\int_{0}^{t} \int_{B_{\rho'}(1)} f d(\mu-\nu)(s,f)}{\phi} = \int_{0}^{t} \int_{B_{\rho'}(1)} \inner{f}{\phi} d(\mu-\nu)(s,f), \quad \forall \, \phi \in \Phi, \, t \geq 0,$$
\item $\left( \int_{0}^{t} \int_{B_{\rho'}(1)^{c}} f d\mu (s,f) : t \geq 0 \right)$ is a $\Phi'$-valued regular  adapted c\`{a}dl\`{a}g process which has $\Prob$-a.a. paths with only a finite number of jumps on each bounded interval of $\R_{+}$ (in particular is a finite variation process) satisfying 
$$ \int_{0}^{t} \int_{B_{\rho'}(1)^{c}} f d\mu(s,f) = \sum_{s \leq t} \Delta X_{s} \mathbbm{1}_{\{ \Delta X_{s} \in B_{\rho'}(1)^{c} \}} , \quad \forall  \, t \geq 0.$$ 
\end{enumerate}
\end{theo}

In order to prove Theorem \ref{theoSemirtDecompNucleSpace} we will need to go through several steps. We benefit from ideas taken from \cite{BadrikianUstunel:1996} and  \cite{PerezAbreu:1988}. 
  
First, from Theorem \ref{theoRegulCylinSemimartigalesNuclear} there exists a weaker countably Hilbertian topology $\vartheta$ on $\Phi$ such that $X$ has an indistinguishable version that is a $(\widehat{\Phi}_{\vartheta})'$-valued adapted c\`{a}dl\`{a}g semimartingale. We will identify $X$ with this version. Moreover, from Proposition \ref{propCylSemimContOpeSpaceSemim} the topology $\vartheta$ can be chosen such that $X$ defines a continuous linear map from $\widehat{\Phi}_{\vartheta}$ into $S^{0}$. Without loss of generality we can assume that $\rho$ is continuous with respect to the topology $\vartheta$. Otherwise, we can just consider another countably Hilbertian topology on $\Phi$ generated by $\rho$ toguether with the countable family of generating seminorms of $\vartheta$. An important consequence of this remark is that the unit ball $B_{\rho'}(1)$ of $\rho'$ is a bounded subset in $(\widehat{\Phi}_{\vartheta})'$. 

Now we show the existence of the different components of the representation \eqref{eqSemimartDecompNucleSpace}. Observe that since  $X$ is adapted, then $X_{0}$ is a $(\widehat{\Phi}_{\vartheta})'$-valued random variable that is $\mathcal{F}_{0}$-measurable. The continuity of the canonical inclusion from $(\widehat{\Phi}_{\theta})'$ into $\Phi'$ shows  that $X_{0}$ satisfies Theorem \ref{theoSemirtDecompNucleSpace}(1).

Now, consider the random measure of the jumps of $X$:  
\begin{equation} \label{defiRandomMeasureJumps}
\mu(\omega; (0,t]; \Gamma)= \sum_{0<s \leq t} \mathbbm{1}_{\{ \Delta X_{s}(\omega) \in \Gamma \}}, \quad \forall \, t \geq 0, \, \Gamma \in \mathcal{B}(\Phi'_{0}).
\end{equation}
Since $\mu(\omega; (0,t]; \cdot)$ has its support on $(\widehat{\Phi}_{\vartheta})'$, we only need to check that it is finite for $\Gamma \in \mathcal{B}((\widehat{\Phi}_{\vartheta})' \setminus \{0\} )$, $0 \notin \overline{\Gamma}$.  
But since the (indistinguishable version of) $X$ satisfies that for $\Prob$-a.e. $\omega \in \Omega$ and $t \geq 0$, there exists a $\vartheta$-continuous Hilbertian semi-norm $\varrho=\varrho(\omega,t)$ on $\Phi$ such that the map $s \mapsto X_{s}(\omega)$ is c\`{a}dl\`{a}g from $[0,t]$ into the Hilbert space $\Phi'_{\varrho}$ (see Remark 3.9 in \cite{FonsecaMora:2018}), and because $\Phi'_{\varrho}$ is a complete separable metric space, then $\Delta X_{s}(\omega) \neq 0$ for a finite number of $s \in [0,t]$. Therefore $\mu(\omega; (0,t]; \Gamma)< \infty $ $\Prob$-a.e. for each $\Gamma \in  \mathcal{B}(\Phi'_{0})$, $0 \notin \overline{\Gamma}$, and hence $\mu$ is a well-defined integer-valued random measure. Furthermore, since $\mu$ can be regarded as a random measure on $\R_{+} \times (\widehat{\Phi}_{\vartheta})'$, and because $((\widehat{\Phi}_{\vartheta})', \mathcal{B}((\widehat{\Phi}_{\vartheta})'))$ is a standard measurable space (see Theorem 2.1.7 in \cite{Ito}), then by VIII.66(b) in \cite{DellacherieMeyer}  $\mu$ admits a (predictable) compensator measure $\nu$ such that for each non-negative predictable function $g=g(\omega,t,f)$:
\begin{equation} \label{eqDefCompenMeasureCountHilTop}
\Exp \int_{0}^{\infty} \int_{(\widehat{\Phi}_{\vartheta})'} g d\mu= \Exp \int_{0}^{\infty} \int_{(\widehat{\Phi}_{\vartheta})'} g d\nu, 
\end{equation}
and 
\begin{gather*}
\nu(\omega; \{ 0\}; (\widehat{\Phi}_{\vartheta})') = \nu(\omega; \R_{+}; \{0\})=0,\\ 
\nu(\omega; \{ t\}; (\widehat{\Phi}_{\vartheta})' ) \leq  1, \, \, \forall t > 0.
\end{gather*}
Moreover, since for each $\phi \in \widehat{\Phi}_{\vartheta}$, $X(\phi)$ is a real-valued semimartingale, then the process $( \sum_{s \leq t} (\abs{\inner{\Delta X_{s}}{\phi}}^{2} \wedge 1): t \geq 0)$ is locally integrable (see \cite{JacodShiryaev}, Theorem I.4.47), hence from \eqref{eqDefCompenMeasureCountHilTop} we have $\forall \phi \in \widehat{\Phi}_{\vartheta}$, $t>0$, $\Prob$-a.e.
$$\int^{t}_{0} \int_{(\widehat{\Phi}_{\vartheta})'} \, \abs{\inner{f}{\phi}}^{2} \wedge 1 \,  \nu (ds, df)< \infty.$$
Since the canonical inclusion from $(\widehat{\Phi}_{\vartheta})'$ into $\Phi'$ is linear and continuous, the compensator measure $\nu$ can be lifted to $\Phi'$ and satisfy $(4)(a)$-$(c)$ in Theorem \ref{theoSemirtDecompNucleSpace}. In a similar way one can show that (see e.g. VIII.68.4 in \cite{DellacherieMeyer}) 
\begin{equation} \label{eqSquareIntegCompRandMeasure}
\int^{t}_{0} \int_{\Phi'} \, \abs{\inner{f}{\phi}}^{2} \wedge \abs{\inner{f}{\phi}} \,  \nu (ds, df)< \infty \quad \forall \, \phi \in \Phi, \, t> 0.  
\end{equation}

Now, from the properties of $\mu$ it is clear that the integral
$$ \int_{0}^{t} \int_{B_{\rho'}(1)^{c}} f d\mu(s,f) = \sum_{s \leq t} \Delta X_{s} \mathbbm{1}_{\{ \Delta X_{s} \in B_{\rho'}(1)^{c} \}} , \quad \forall \, t \geq 0,$$ 
is a $(\widehat{\Phi}_{\vartheta})'$-valued adapted c\`{a}dl\`{a}g process which has $\Prob$-a.a. paths with only a finite number of jumps on each bounded interval of $\R_{+}$. Since the canonical inclusion from $(\widehat{\Phi}_{\theta})'$ into $\Phi'$ is linear and continuous, then $\left( \int_{0}^{t} \int_{B_{\rho'}(1)^{c}} f d\mu (s,f) : t \geq 0 \right)$ as a $\Phi'$-valued process satisfying Theorem \ref{theoSemirtDecompNucleSpace}(6). 

Define $Y=(Y_{t}: t \geq 0)$ by the prescription 
\begin{equation} \label{eqDefDecoSpecSemiBoundJumps}
Y_{t}=X_{t}-X_{0} - \int_{0}^{t} \int_{B_{\rho'}(1)^{c}} f d\mu(s,f), \quad \forall t \geq 0.
\end{equation}
Then, $Y$ is a $(\widehat{\Phi}_{\vartheta})'$-valued adapted c\`{a}dl\`{a}g process with $Y_{0}=0$.  Moreover, we have the following: 

\begin{lemm}\label{lemmSpecialSemimDecompBounded}
The process $Y=(Y_{t}:t \geq 0)$ admits a (unique up to indistinguishable versions) representation  
\begin{equation}\label{eqSpecialSemimartDecompNucleSpace}
Y_{t}=M^{c}_{t}+M^{d}_{t}+A_{t}, \quad \forall t \geq 0,
\end{equation}
where $M^{c}=(M^{c}_{t}: t \geq 0)$ is a $(\widehat{\Phi}_{\vartheta})'$-valued continuous local martingale with $M^{c}_{0}=0$, $M^{d}=(M^{d}_{t}: t \geq 0)$ is a $(\widehat{\Phi}_{\vartheta})'$-valued  purely discontinuous local martingale and $M^{d}_{0}=0$, $A=(A_{t}: t \geq 0)$ is a $(\widehat{\Phi}_{\vartheta})'$-valued predictable c\`{a}dl\`{a}g process of locally integrable variation and $A_{0}=0$. 
\end{lemm}
\begin{prf}
First, observe that by construction the set of jumps $\{ \Delta Y_{t}: t \geq 0\}$ of $Y$ is contained in the bounded subset $B_{\rho'}(1)$ in $(\widehat{\Phi}_{\vartheta})'$, and consequently $\{ \Delta Y_{t}: t \geq 0\}$ is itself bounded in $(\widehat{\Phi}_{\vartheta})'$. Therefore, the definition of strong boundedness implies that for every bounded subset $C$ in $\widehat{\Phi}_{\vartheta}$ there exists a $K_{C}>0$ such that 
\begin{equation} \label{eqUnifBoundJumpsSpeciSemi}
\sup_{\phi \in C} \sup_{t \geq 0} \abs{\inner{\Delta Y_{t}}{\phi}} < K_{C}. 
\end{equation}
However, since for each $\phi \in \widehat{\Phi}_{\vartheta}$ the set $\{ \phi \}$ is bounded in $\widehat{\Phi}_{\vartheta}$, then \eqref{eqUnifBoundJumpsSpeciSemi} shows that the real-valued semimartingale $\inner{Y}{\phi}$ has uniformly bounded jumps,  hence is a special semimartingale and has a (unique) representation (see the proof of Theorem III.35 in \cite{Protter}, p.131)
$$ \inner{Y_{t}}{\phi}=M^{\phi}_{t}+A^{\phi}_{t}, \quad \forall t \geq 0, $$
where $(M^{\phi}_{t}:t \geq 0)$ is a real-valued c\`{a}dl\`{a}g local martingale, $(A_{t}^{\phi}:t \geq 0)$ is a real-valued predictable c\`{a}dl\`{a}g process  of locally integrable variation,  and $M^{\phi}_{0}=A^{\phi}_{0}=0$. Furthermore, each $M^{\phi}$ has a (unique) representation (see \cite{DellacherieMeyer}, Theorem VIII.43, p.353)
$$ M^{\phi}_{t}=M^{c,\phi}_{t}+M^{d,\phi}_{t}, \quad \forall t \geq 0, $$
where $(M^{c,\phi}_{t}:t \geq 0)$ is a real-valued continuous local martingale, $(M^{d,\phi}_{t}:t \geq 0)$ is a real-valued purely discontinuous local martingale; these two local martingales are orthogonal. Therefore, $\inner{Y}{\phi}$ has the (unique) representation  
\begin{equation}\label{eqFiniDimDecompSpecSemBoun}
\inner{Y_{t}}{\phi}=M^{c,\phi}_{t}+M^{d,\phi}_{t}+A^{\phi}_{t}, \quad \forall t \geq 0. 
\end{equation}

If for every $t \geq 0$ we define the mappings $\tilde{M}^{c}_{t}:\widehat{\Phi}_{\vartheta} \rightarrow L^{0} \ProbSpace$, $\phi \mapsto \tilde{M}^{c}_{t}(\phi)\defeq M^{c,\phi}_{t}$, $\tilde{M}^{d}_{t}:\widehat{\Phi}_{\vartheta} \rightarrow L^{0} \ProbSpace$, $\phi \mapsto \tilde{M}^{d}_{t}(\phi)\defeq M^{d,\phi}_{t}$, and $\tilde{A}_{t}:\widehat{\Phi}_{\vartheta} \rightarrow L^{0} \ProbSpace$, $\phi \mapsto \tilde{A}_{t}(\phi)\defeq A^{\phi}_{t}$, then by uniqueness of the decomposition (see the arguments in the proof of Theorem 3.1 in \cite{BadrikianUstunel:1996}) the maps $\tilde{M}^{c}_{t}$, $\tilde{M}^{d}_{t}$ and $\tilde{A}_{t}$ are linear. Therefore, $\tilde{M}^{c}=(\tilde{M}^{c}_{t}: t \geq 0)$, $\tilde{M}^{d}=(\tilde{M}^{d}_{t}: t \geq 0)$  and $\tilde{A}=(\tilde{A}_{t}: t \geq 0)$ are cylindrical semimartingales in $(\widehat{\Phi}_{\vartheta})'$ and hence define linear maps from $\widehat{\Phi}_{\vartheta}$ into $S^{0}$. Our next objective is to show they are also continuous. Since $\widehat{\Phi}_{\vartheta}$ is ultrabornological, it is enough to show that the maps $\tilde{M}^{c}$, $\tilde{M}^{d}$ and $\tilde{A}$ are sequentially closed, because in that case the closed graph theorem shows that they are continuous (\cite{NariciBeckenstein}, Theorem 14.7.3, p.475). 

Let $\phi_{n} \rightarrow \phi$ in $\widehat{\Phi}_{\vartheta}$ and suppose that $\tilde{M}^{c}(\phi_{n}) \rightarrow m^{c}$, $\tilde{M}^{d}(\phi_{n}) \rightarrow m^{d}$ and $\tilde{A}(\phi_{n}) \rightarrow a$ in $S^{0}$. By the continuity of  $Y$ from $\widehat{\Phi}_{\vartheta}$ into $S^{0}$, we have $Y(\phi_{n}) \rightarrow Y(\phi)$ in $S^{0}$. 
Since the set $C=\{ \phi_{n}: n \in \N \}$ is bounded in $\widehat{\Phi}_{\theta}$, then from \eqref{eqUnifBoundJumpsSpeciSemi} the family of real-valued semimartingales $\{ \inner{Y}{\phi_{n}}: n \in \N\}$ has jumps uniformly bounded by $K_{C}$. But as the collection of all the real-valued semimartingales with jumps uniformly bounded by $K_{C}$ is a closed subspace in $S^{0}$ (see \cite{Memin:1980}, Theorem IV.4), then $\inner{Y}{\phi}$ also has jumps uniformly bounded by $K_{C}$. Therefore, because $Y(\phi_{n}) \rightarrow Y(\phi)$ in $S^{0}$ we have that (see Remarque IV.3 in \cite{Memin:1980}) $\tilde{M}^{c}(\phi_{n}) \rightarrow \tilde{M}^{c}(\phi)$, $\tilde{M}^{d}(\phi_{n}) \rightarrow \tilde{M}^{d}(\phi)$ and $\tilde{A}(\phi_{n}) \rightarrow \tilde{A}(\phi)$ in $S^{0}$. Hence, by uniqueness of limits we get that $m^{c}=\tilde{M}^{c}(\phi)$, $m^{d}=\tilde{M}^{d}(\phi)$ and $a=\tilde{A}(\phi)$. Thus, the mappings $\tilde{M}^{c}$, $\tilde{M}^{d}$ and $\tilde{A}$ are continuous from $\widehat{\Phi}_{\vartheta}$ into $S^{0}$. 

Hence, from Propositions     \ref{propCylSemimContOpeSpaceSemim} and \ref{propRegulaLocalMarBoundVaria}, there exists another weaker countably Hilbertian topology $\theta$ on $\Phi$, larger than $\vartheta$, and a $(\widehat{\Phi}_{\theta})'$-valued continuous local martingale $M^{c}=(M^{c}_{t}: t \geq 0)$,  a $(\widehat{\Phi}_{\theta})'$-valued purely discontinuous local martingale $M^{d}=(M^{d}_{t}: t \geq 0)$, and a $(\widehat{\Phi}_{\theta})'$-valued predictable c\`{a}dl\`{a}g process of locally integrable variation $A=(A_{t}: t \geq 0)$ with $A_{0}=0$, all of them such that $\Prob$-a.e. $\forall t \geq 0$, $\phi \in \Phi$,  
\begin{gather}
 \inner{M^{c}_{t}}{\phi}=\tilde{M}^{c}_{t}(\phi), \label{eqDefPartMcDecomp} \\
\inner{M^{d}_{t}}{\phi}=\tilde{M}^{d}_{t}(\phi), \label{eqDefPartMdDecomp} \\ 
\inner{A_{t}}{\phi}=\tilde{A}_{t}(\phi) \label{eqDefPartADecomp}.
\end{gather}
But then, \eqref{eqFiniDimDecompSpecSemBoun}, \eqref{eqDefPartMcDecomp}, \eqref{eqDefPartMdDecomp} and \eqref{eqDefPartADecomp} imply that $\Prob$-a.e. $\forall t \geq 0$, $\phi \in \Phi$,
\begin{equation}\label{eqPointSpecSemiDecomp}
\inner{Y_{t}}{\phi}=\inner{M^{c}_{t}}{\phi}+ \inner{M^{d}_{t}}{\phi} +\inner{A_{t}}{\phi}.
\end{equation}
Now, since the processes $Y$, $M^{c}$, $M^{d}$ and $A$ are regular processes, then \eqref{eqPointSpecSemiDecomp} together with Proposition 2.12 in \cite{FonsecaMora:2018} implies that $Y$, $M^{c}$, $M^{d}$ and $A$ satisfy \eqref{eqSpecialSemimartDecompNucleSpace}. 
\end{prf}

\begin{lemm}\label{lemmUnifBoundJumpsDecomSpecSemi}
The processes $M^{d}=(M^{d}_{t}: t \geq 0)$ and $A=(A_{t}: t \geq 0)$ defined in Lemma \ref{lemmSpecialSemimDecompBounded} have $\Prob$-a.e. uniformly bounded jumps in $\widehat{\Phi}_{\theta}$.  
\end{lemm}
\begin{prf}
Let $\phi \in \widehat{\Phi}_{\theta}$. By the definition of local martingale and of process of integrable variation, there exists a sequence of stopping times $(\tau_{n}:n \in \N)$ increasing to $\infty$ such that for all $n \in \N$, $(\inner{M^{c}_{t \wedge \tau_{n}}}{\phi}: t \geq 0)$ and $(\inner{M^{d}_{t \wedge \tau_{n}}}{\phi}: t \geq 0)$ are uniformly integrable martingales and $(\inner{A_{t \wedge \tau_{n}}}{\phi}: t \geq 0)$ is of integrable total variation. Moreover, since $(\inner{A_{t}}{\phi}: t \geq 0)$ is predictable and $(\inner{M^{c}_{t}}{\phi}: t \geq 0)$ is continuous, it then follows from \eqref{eqPointSpecSemiDecomp} that $\Prob$-a.e. $\forall t \geq 0$,
\begin{equation} \label{eqCondExpJumpsY}
\Exp \left( \inner{\Delta Y_{t\wedge \tau_{n}}}{\phi} \vline \mathcal{F}_{t^{-}} \right) = \Exp \left( \inner{\Delta M^{d}_{t \wedge \tau_{n}}}{\phi}+\inner{\Delta A_{t \wedge \tau_{n}}}{\phi} | \mathcal{F}_{t^{-}} \right)= \inner{\Delta A_{t \wedge \tau_{n}}}{\phi}. 
\end{equation}
Now, recall that the set of jumps $\{ \Delta Y_{t}: t \geq 0\}$ of $Y$ is bounded in $(\widehat{\Phi}_{\vartheta})'$, but since the topology $\theta$ is finer than $\vartheta$ the canonical inclusion from  $(\widehat{\Phi}_{\vartheta})'$ into $(\widehat{\Phi}_{\theta})'$ is continuous; then  $\{ \Delta Y_{t}: t \geq 0\}$ is bounded in $(\widehat{\Phi}_{\theta})'$. Hence, from \eqref{eqUnifBoundJumpsSpeciSemi} with $C=\{\phi\}$ and taking limits as $n \rightarrow \infty$ in \eqref{eqCondExpJumpsY}, we have that $\Prob$-a.e. $\forall t \geq 0$, 
$$ \Exp \left( \inner{\Delta Y_{t}}{\phi} | \mathcal{F}_{t^{-}} \right) = \inner{\Delta A_{t}}{\phi}. $$ 
Therefore, for any given bounded subset $C$ in $\widehat{\Phi}_{\theta}$ (recall that this space is separable), it follows from \eqref{eqUnifBoundJumpsSpeciSemi} that $\Prob$-a.e.
$$ \sup_{t \geq 0} \sup_{\phi \in C} \abs{\inner{\Delta A_{t}}{\phi}} < K_{C}.$$
But the above inequality together with \eqref{eqPointSpecSemiDecomp} shows that 
$\Prob$-a.e.
$$  \sup_{t \geq 0} \sup_{\phi \in C} \abs{\inner{\Delta M^{d}_{t}}{\phi}} < 2 K_{C}.$$
Hence, the processes $M^{d}$ and $A$ have $\Prob$-a.e. uniformly bounded jumps in $\widehat{\Phi}_{\theta}$.
\end{prf}

\begin{lemm}\label{lemmDiscLocaMartAsInteg}
For each $t >0$ and $\phi \in \Phi$, 
\begin{equation} \label{eqDefDiscoMartPartAsIntegral}
\inner{M^{d}_{t}}{\phi} = \int_{0}^{t} \int_{B_{\rho'}(1)} \inner{f}{\phi} d(\mu-\nu)(s,f).
\end{equation}
\end{lemm}
\begin{prf}
First, for each $\phi \in \Phi$ it is a  consequence of \eqref{eqSquareIntegCompRandMeasure} (see Theorem 2.1 in \cite{KabanovLipcerSirjaev:1979}) that $\left( \int_{0}^{t} \int_{B_{\rho'}(1)} \inner{f}{\phi} d(\mu-\nu)(s,f): t \geq 0 \right)$ is a purely discontinuous local martingale satisfying 
$$ \Delta \left( \int_{0}^{t} \int_{B_{\rho'}(1)} \inner{f}{\phi} d(\mu-\nu)(s,f) \right)=\int_{B_{\rho'}(1)} \inner{f}{\phi} \mu(\{t\},f)-\int_{B_{\rho'}(1)} \inner{f}{\phi} \nu(\{t\},f). $$
Thus, since $(\inner{M^{d}_{t}}{\phi}:t \geq 0)$ is also a purely discontinuous local martingale, to prove \eqref{eqDefDiscoMartPartAsIntegral} it is enough to show that 
$$ \Delta \inner{M^{d}_{t}}{\phi} = \int_{B_{\rho'}(1)} \inner{f}{\phi} \mu(\{t\},f)-\int_{B_{\rho'}(1)} \inner{f}{\phi} \nu(\{t\},f). $$
But the above equality follows from exactly the same arguments to those used in the proof of Theorem 3 in \cite{PerezAbreu:1988}, so we leave the details to the reader. 
\end{prf}

To finallize the proof of Theorem \ref{theoSemirtDecompNucleSpace}, observe that since the canonical inclusion from $(\widehat{\Phi}_{\theta})'$ into $\Phi'$ is linear and continuous, then $A$, $M^{c}$ and $M^{d}$ define $\Phi'$-valued processes. From Lemma \ref{lemmSpecialSemimDecompBounded} we have that $M^{c}$ satisfies  Theorem \ref{theoSemirtDecompNucleSpace}(3). For $A$, it follows from Lemmas \ref{lemmSpecialSemimDecompBounded} and \ref{lemmUnifBoundJumpsDecomSpecSemi}, and Theorem \ref{theoRegulCylSemIntegBouVari}, that $A$ has an indistinguishable version that satisfies Theorem \ref{theoSemirtDecompNucleSpace}(2). If for each 
$t \geq 0$ we denote $M^{d}_{t}$ by $\int_{0}^{t} \int_{B_{\rho'}(1)} f d(\mu-\nu)(s,f)$, then it is a consequence of  Lemmas \ref{lemmSpecialSemimDecompBounded},  \ref{lemmUnifBoundJumpsDecomSpecSemi} and \ref{lemmDiscLocaMartAsInteg} that $\left( \int_{0}^{t} \int_{B_{\rho'}(1)} f d(\mu-\nu)(s,f): t \geq 0 \right)$ satisfies Theorem \ref{theoSemirtDecompNucleSpace}(5). Finally, the fact that $X$ admits the representation \eqref{eqSemimartDecompNucleSpace} follows from \eqref{eqDefDecoSpecSemiBoundJumps} and \eqref{eqSpecialSemimartDecompNucleSpace}.

\begin{rema}\label{uniqueContMartAndCompMeasu} In the proof of Theorem \ref{theoSemirtDecompNucleSpace} we have used a $(\widehat{\Phi}_{\vartheta})'$-valued indistinguishable version of $X$, which at a first glance might lead us to conclude that the random measure $\mu$ of the jumps of $X$ depends on the countably Hilbertian topology $\vartheta$ on $\Phi$. However, this is not the case as the aforementioned version of $X$ is unique (up to indistinguishable versions) as a $\Phi'$-valued process (see Theorem \ref{theoRegulCylinSemimartigalesNuclear}). Therefore, 
the random measure $\mu$ of the jumps of $X$ defined in \eqref{defiRandomMeasureJumps} is unique and consequently its compensated measure $\nu$. Hence $\mu$ and $\nu$ do not  depend on the given seminorm $\rho$. Similarly, the continuous local martingale part $M^{c}$ is independent of the given seminorm $\rho$. To prove this, suppose $M^{1,c}$ and $M^{2,c}$ denote the continuous local martingale part of any two representations of $X$ as in \eqref{eqSemimartDecompNucleSpace}. For any given $\phi \in \Phi$, let $X^{c,\phi}$ denotes the real-valued continuous local martingale corresponding to the canonical representation of $\inner{X}{\phi}$. Then, by uniqueness of $X^{c,\phi}$ we have $\Prob$-a.e.
$$\inner{M^{1,c}_{t}}{\phi} = X_{t}^{c,\phi} =\inner{M^{2,c}_{t}}{\phi}, \quad \forall t \geq 0.$$
But since $M^{1,c}$ and $M^{2,c}$ are regular processes, Proposition 2.12 in \cite{FonsecaMora:2018} shows that $M^{1,c}= M^{2,c}$. 
\end{rema}

Under additional assumptions on  $\Phi$, the next result shows that the representation \eqref{eqSemimartDecompNucleSpace} remains valid for any $\Phi'$-valued semimartingale:  

\begin{prop} \label{propSemimaDecompUltrabNucleSpace}
The conclusion of Theorem \ref{theoSemirtDecompNucleSpace} remains valid if $\Phi$ is an ultrabornological nuclear space and $X=(X_{t}: t \geq 0)$ is a $\Phi'$-valued c\`{a}dl\`{a}g semimartingale such that the probability distribution of each $X_{t}$ is a Radon measure.
\end{prop}
\begin{prf}
The result follows if in the proof of Theorem \ref{theoSemirtDecompNucleSpace} we use Corollary \ref{coroRegulCylSemiUltraborno} instead of Theorem \ref{theoRegulCylinSemimartigalesNuclear} and if instead of Proposition \ref{propCylSemimContOpeSpaceSemim} we use Corollary 
\ref{coroCylSemimContOpeSpaceSemimUltra}. 
\end{prf}

\begin{coro} \label{coroDecomSemimarFrecNucleSpace}
If $\Phi$ is a Fr\'{e}chet nuclear space or the countable inductive limit of Fr\'{e}chet nuclear spaces, then each 
$\Phi'$-valued c\`{a}dl\`{a}g semimartingale $X=(X_{t}: t \geq 0)$ possesses a representation satisfying the conditions in Theorem \ref{theoSemirtDecompNucleSpace}.
\end{coro}
\begin{prf}
When $\Phi$ is a Fr\'{e}chet nuclear space or a countable inductive limit of Fr\'{e}chet nuclear spaces, then every Borel measure on $\Phi'_{\beta}$ is a Radon measure (see Corollary 1.3 of Dalecky and Fomin \cite{DaleckyFomin}, p.11). In particular, for each $t \geq 0$ the probability distribution of $X_{t}$ is Radon. The result now follows from   Proposition \ref{propSemimaDecompUltrabNucleSpace}. 
\end{prf} 

The representation in Theorem \ref{theoSemirtDecompNucleSpace} leads naturaly to the following definition:

\begin{defi}
Let $X=(X_{t}:t \geq 0)$ be $\Phi'$-valued adapted c\`{a}dl\`{a}g  semimartingale such that for each $T>0$,  the family of linear maps $( X_{t}: t \in [0,T])$ from $\Phi$ into $L^{0}\ProbSpace$ is equicontinuous (at the origin). Given a continuous Hilbertian seminorm $\rho$ on $\Phi$, we call \emph{characteristics of $X$} relative to $\rho$ the triplet $(A,C,\nu)$ consisting in:
\begin{enumerate}
\item $A$ is the $\Phi'$-valued process defined in Theorem \ref{theoSemirtDecompNucleSpace}(2). 
\item  $C: \Omega \times \R_{+} \times \Phi \times \Phi$ is the map defined by 
$$ C(\omega, t, \phi, \varphi)= \llangle \, \inner{M^{c}}{\phi} , \inner{M^{c}}{\varphi} \ \rrangle_{t}(\omega), \quad \forall \, (\omega, t, \phi, \varphi) \in  \Omega \times \R_{+} \times \Phi \times \Phi,$$
where $M^{c}$ is as given in Theorem \ref{theoSemirtDecompNucleSpace}(3), and $\llangle \, \inner{M^{c}}{\phi} , \inner{M^{c}}{\varphi} \ \rrangle$ is the angle bracket process of the continuous local martingales $\inner{M^{c}}{\phi}$ and $\inner{M^{c}}{\varphi} $. 
\item $\nu$ is the (predictable) compensator measure of the random measure $\mu$ associated to the jumps of $X$, as given in Theorem \ref{theoSemirtDecompNucleSpace}(4).
\end{enumerate}
\end{defi}

\begin{rema} It follows from Remark \ref{uniqueContMartAndCompMeasu} that $C$ and $\nu$ do not depend on the choice of the seminorm $\rho$ on $\Phi$, while $A=A(\rho)$ does.
\end{rema}

\section{Characteristics and L\'{e}vy Processes}\label{sectCharLevyProce} 

In this section we study in detail the canonical decomposition and characteristics of a $\Phi'$-valued L\'{e}vy process and how they relate with their \emph{L\'{e}vy-It\^{o} decomposition} studied in \cite{FonsecaMora:Levy}.  

Let $L=( L_{t}: t \geq 0)$ be a $\Phi'$-valued c\`{a}dl\`{a}g L\'{e}vy process. Suppose that for every $T > 0$, the family $( L_{t}: t \in [0,T] )$ of linear maps from $\Phi$ into $L^{0} \ProbSpace$  is equicontinuous (at the origin) (see Example \ref{examCylLevyRegulari}).   
Under the above conditions it is proved in Section 4 in \cite{FonsecaMora:Levy} that the random measure $\mu$ of the jumps of $L$ is a Poisson Random measure, that we denote by $N$, and that its (predictable) compensator measure is of the form $\nu(\omega; dt; df)= dt \nu (df)$, where $\nu$ is a \emph{L\'{e}vy measure} on $\Phi'$ in the following sense (see \cite{FonsecaMora:Levy}, Theorem 4.11):
\begin{enumerate}
\item $\nu (\{ 0 \})=0$, 
\item for each neighborhood of zero $U \subseteq \Phi'$, the  restriction $\restr{\nu}{U^{c}}$ of $\nu$ on the set $U^{c}$ belongs to the space $\goth{M}^{b}_{R}(\Phi')$ of bounded Radon measures on $\Phi'$,    
\item there exists a continuous Hilbertian semi-norm $\rho$ on $\Phi$ such that 
\begin{equation} \label{integrabilityPropertyLevyMeasure}
\int_{B_{\rho'}(1)} \rho'(f)^{2} \nu (df) < \infty,  \quad \mbox{and} \quad  \restr{\nu}{B_{\rho'}(1)^{c}} \in \goth{M}^{b}_{R}(\Phi'), 
\end{equation}
where recall that $\rho'$ is the dual norm of $\rho$ (see Sect. \ref{secNotaDefi}) and $B_{\rho'}(1) \defeq B_{\rho}(1)^{0}=\{ f \in \Phi': \rho'(f) \leq 1\}$. 
\end{enumerate}
Here is important to stress the fact that the seminorm $\rho$ satisfying \eqref{integrabilityPropertyLevyMeasure} is not unique. Indeed, any continuous Hilbert semi-norm $q$ on  such that $\rho \leq q$ satisfies \eqref{integrabilityPropertyLevyMeasure}.

It is shown in Theorem 4.17 in \cite{FonsecaMora:Levy} that relative to a continuous Hilbertian seminorm $\rho$ on $\Phi$ satisfying \eqref{integrabilityPropertyLevyMeasure},  for each $t \geq 0$, $L_{t}$ admits the unique representation
\begin{equation} \label{levyItoDecomposition}
L_{t}=t\goth{m}+W_{t}+\int_{B_{\rho'}(1)} f \widetilde{N} (t,df)+\int_{B_{\rho'}(1)^{c}} f N (t,df)
\end{equation}
that is usually called the \emph{L\'{e}vy-It\^{o} decomposition} of $L$. In \eqref{levyItoDecomposition}, we have that $\goth{m} \in \Phi'$, $\widetilde{N}(dt,df)= N(dt,df)-dt \, \nu(df)$ is the compensated Poisson random measure, and $( W_{t} : t \geq 0)$ is a $\Phi'$-valued L\'{e}vy process with continuous paths (also called a $\Phi'$-valued Wiener process) with zero-mean (i.e. $\Exp (\inner{W_{t}}{\phi}) = 0$ for each $t \geq 0$ and $\phi \in \Phi$) and \emph{covariance functional} $\mathcal{Q}$ satisfying 
\begin{equation}\label{covarianceFunctWienerProcess}
\Exp \left( \inner{W_{t}}{\phi} \inner{W_{s}}{\varphi} \right) = ( t \wedge s ) \mathcal{Q} (\phi, \varphi), \quad \forall \, \phi, \varphi \in \Phi, \, s, t \geq 0. 
\end{equation}
Observe that $\mathcal{Q}$ is a continuous, symmetric, non-negative bilinear form on $\Phi \times \Phi$. It is important to remark that all the random components of the representation \eqref{levyItoDecomposition} are independent. 

Observe that when we compare \eqref{levyItoDecomposition} with \eqref{eqSemimartDecompNucleSpace}, we conclude that $A_{t}=t \goth{m}$, $M_{t}^{c}=W_{t}$, and that $\int_{0}^{t} \int_{B_{\rho'}(1)} f d(\mu-\nu)(s,f)$ and  $\int_{0}^{t} \int_{B_{\rho'}(1)^{c}} f d\mu (s,f)$ corresponds to the Poisson integrals $\int_{B_{\rho'}(1)} f \widetilde{N} (t,df)$ and $\int_{B_{\rho'}(1)^{c}} f N (t,df)$ (for details on the definition of Poisson integrals see \cite{FonsecaMora:Levy}). 
Moreover, since for each $\phi \in \Phi$ we have that $\inner{W}{\phi}$ is a real-valued Wiener process, then (see Theorem II.4.4 in \cite{JacodShiryaev}) it follows from \eqref{covarianceFunctWienerProcess} that 
$$ C(\omega, t, \phi, \varphi)= \llangle \, \inner{W}{\phi} , \inner{W}{\varphi} \ \rrangle_{t}(\omega)=\Exp \left( \inner{W_{t}}{\phi} \inner{W_{t}}{\varphi} \right) = t \mathcal{Q} (\phi, \varphi), $$ for all $ (\omega, t, \phi, \varphi) \in  \Omega \times \R_{+} \times \Phi \times \Phi$.
Therefore, we conclude that the characteristics of $L$ relative to $\rho$ satisfying  \eqref{integrabilityPropertyLevyMeasure} are deterministic and have the form:
\begin{equation} \label{eqCharactOfLevyProcess}
A_{t}=t \goth{m}, \quad  C(\omega, t, \phi, \varphi)=t \mathcal{Q} (\phi, \varphi), \quad \nu(\omega; dt; df)= dt \, \nu (df). 
\end{equation}

The next result shows that the above form of the characteristics is exclusive of the $\Phi'$-valued L\'{e}vy processes: 

\begin{theo}\label{theoCharactLevyProc}
Let $L=(L_{t}:t \geq 0)$ be $\Phi'$-valued, adapted, c\`{a}dl\`{a}g  process such that for each $T>0$,  the family of linear maps $( L_{t}: t \in [0,T])$ from $\Phi$ into $L^{0}\ProbSpace$ is equicontinuous (at the origin). Then, $L$ is a $\Phi'$-valued L\'{e}vy process if and only if it is a $\Phi'$-valued semimartingale and there exists a continuous Hilbertian semi-norm $\rho$ on $\Phi$ such that the  characteristics of $L$ relative to $\rho$ are of the form  \eqref{eqCharactOfLevyProcess}, where $\goth{m} \in \Phi'$, $\mathcal{Q}$ is a continuous, symmetric, non-negative bilinear form on $\Phi \times \Phi$, and $\nu$ is a L\'{e}vy measure on $\Phi$ for which $\rho$ satisfy  \eqref{integrabilityPropertyLevyMeasure}. Moreover, for all
$t \geq 0$, $\phi \in \Phi$, 
\begin{equation} \label{levyKhintchineFormulaLevyProcess}
\begin{split}
& \Exp \left( e^{i \inner{L_{t}}{\phi} } \right) = e^{t \eta(\phi)}, \quad  \mbox{ with} \\ 
& \eta(\phi)= i \inner{\goth{m}}{\phi} - \frac{1}{2} \mathcal{Q}(\phi,\phi) + \int_{\Phi'} \left( e^{i \inner{f}{\phi}} -1 - i \inner{f}{\phi} \ind{ B_{\rho'}(1)}{f} \right) \nu(d f).  
\end{split}
\end{equation}
\end{theo}
\begin{proof}
We have already proved that if $L$ is a $\Phi'$-valued L\'{e}vy process then its characteristics relative to $\rho$ are of the form \eqref{eqCharactOfLevyProcess}. Moreover, in that case  \eqref{levyKhintchineFormulaLevyProcess} corresponds to its L\'{e}vy-Khintchine formula (see Theorem 4.18 in  \cite{FonsecaMora:Levy}). 

Assume now that $L$ is a $\Phi'$-valued semimartingale with canonical representation \eqref{eqSemimartDecompNucleSpace} and such that the characteristics of $L$ with respect to $\rho$ are of the form  \eqref{eqCharactOfLevyProcess}. We will show that $L$ induces a cylindrical L\'{e}vy process in $\Phi'$. 

Let $\phi_{1}, \dots, \phi_{n} \in \Phi$ and let $L(\phi_{1}, \dots, \phi_{n})=(L_{t}(\phi_{1}, \dots, \phi_{n}): t\geq 0)$ defined by 
$$ L_{t}(\phi_{1}, \dots, \phi_{n}) \defeq
(\inner{L_{t}}{\phi_{1}}, \dots, \inner{L_{t}}{\phi_{n}}), \quad \forall t \geq 0. $$
From the corresponding properties of $L$, it is clear that   $L(\phi_{1}, \dots, \phi_{n})$ is a $\R^{n}$-valued c\`{a}dl\`{a}g adapted semimartingale. We now study the form of its characteristics $(a,c,\lambda)$ (Definition II.2.6 in \cite{JacodShiryaev})  relative to the truncation function $h(x)=x \mathbbm{1}_{B}$, where $B= \pi_{\phi_{1}, \dots, \phi_{n}}(B_{\rho'}(1))$ and $\pi_{\phi_{1}, \dots, \phi_{n}}: \Phi' \rightarrow \R^{n}$ is the projection 
$$\pi_{\phi_{1}, \dots, \phi_{n}}(f)=(\inner{f}{\phi_{1}}, \dots, \inner{f}{\phi_{n}}), \quad \forall \, f \in \Phi'.$$ 

We study first $c=(c_{ij})_{i,j \leq n}$. Let $M^{c}(\phi_{1}, \dots, \phi_{n})$ denotes the continuous local martingale part of the $\R^{n}$-valued process $L(\phi_{1}, \dots, \phi_{n})$. From \eqref{eqSemimartDecompNucleSpace} and the uniqueness of the continuous martingale part it follows that:
\begin{equation} \label{eqEqualiContPartCylLevy}
M^{c}_{t}(\phi_{1}, \dots, \phi_{n})=(\inner{M^{c}_{t}}{\phi_{1}}, \dots, \inner{M^{c}_{t}}{\phi_{n}}), \quad \forall \, t \geq 0.
\end{equation}
Then, from the definition of $c=(c_{ij})_{i,j \leq n}$, \eqref{eqCharactOfLevyProcess} and the corresponding properties of $\mathcal{Q}$, for each $i, j \leq n$ we have
\begin{equation} \label{eqDefiConQuaVarPartCylLevy}
 c_{i,j}(t,\omega) = \llangle \, \inner{M^{c}}{\phi_{i}} , \inner{M^{c}}{\phi_{j}} \ \rrangle_{t}(\omega) = t q_{i,j}, \quad   \forall \, \omega \in \Omega, \, t \geq 0, 
\end{equation}
where $q=(q_{ij})_{i,j \leq n}$, defined by $q_{i,j} = \mathcal{Q} (\phi_{i}, \phi_{j})$, is a symmetric non-negative $n \times n$ matrix. 

Now, $\lambda$ is the predictable compensator measure of the random measure $\xi$ of the jumps of $L(\phi_{1}, \dots, \phi_{n})$. For a set $D \in \mathcal{B}(\R^{n}_{0})$ ($\R^{n}_{0}= \R^{n}  \setminus \{ 0\}$),  we have that  
\begin{eqnarray}
\xi(\omega; (0,t]; D)
& = &  \sum_{0 \leq s \leq t} \ind{D}{\Delta L_{s}(\phi_{1}, \dots, \phi_{n})(\omega)} \nonumber \\
& = & \sum_{0 \leq s \leq t} \ind{\pi_{\phi_{1}, \dots, \phi_{n}}^{-1}(D)}{\Delta L_{s}(\omega)}  \nonumber \\
& = & \mu(\omega ; (0,t], \pi_{\phi_{1}, \dots, \phi_{n}}^{-1}(D)).   \label{eqEqualiRandMeasuCylLevy}
\end{eqnarray}
But since the predictable compensated measure is unique,  we have that for each $D \in \mathcal{B}(\R^{n}_{0})$, 
\begin{equation} \label{eqEqualiCompMeasuCylLevy}
 \lambda(\omega; t; D)= \nu(\omega; t; \pi_{\phi_{1}, \dots, \phi_{n}}^{-1}(D)). 
\end{equation}
Then, from \eqref{eqCharactOfLevyProcess} we have that 
\begin{equation}  \label{eqDefiCompMeasuCylLevy}
 \lambda(\omega; dt; dy)= dt \, \nu \circ \pi_{\phi_{1}, \dots, \phi_{n}}^{-1}(dy).   
\end{equation} 
Moreover, $\nu \circ \pi_{\phi_{1}, \dots, \phi_{n}}^{-1}$ is a L\'{e}vy measure on $\R^{n}$ since  Theorem \ref{theoSemirtDecompNucleSpace}(4)(c) implies that 
\begin{eqnarray*}
t \int_{\R^{n}} \, \abs{y}^{2} \wedge 1 \, \nu \circ \pi_{\phi_{1}, \dots, \phi_{n}}^{-1}(dy) & = & \int_{0}^{t} \int_{\R^{n}} \abs{y}^{2} \wedge 1 \, ds \, \nu \circ \pi_{\phi_{1}, \dots, \phi_{n}}^{-1}(dy) \\
& = & \int_{0}^{t} \int_{\Phi'} \abs{\pi_{\phi_{1}, \dots, \phi_{n}}(f)}^{2} \wedge 1 \, \nu(ds, df)< \infty. 
\end{eqnarray*}
We now study the $\R^{n}$-valued predictable finite variation process $a=(a_{t}: t \geq 0)$. First, observe that \eqref{eqEqualiRandMeasuCylLevy} and \eqref{eqEqualiCompMeasuCylLevy} show that for all  $t \geq 0$, 
\begin{gather}
 \left( \inner{\int_{0}^{t} \int_{B_{\rho'}(1)} f \,  d(\mu-\nu)(s,f)}{\phi_{1}}, \dots , \inner{\int_{0}^{t} \int_{B_{\rho'}(1)} f \,  d(\mu-\nu)(s,f)}{\phi_{n}} \right) 
 \nonumber \\
  = \int_{0}^{t} \int_{B} y \,  d(\xi-\lambda)(s,y), \label{eqEquCompRandoIntegCylLevy}
\end{gather} 
and 
\begin{gather}
\left(  \inner{\int_{0}^{t} \int_{B_{\rho'}(1)^{c}} f \,  d\mu(s,f)}{\phi_{1}}, \dots, \inner{\int_{0}^{t} \int_{B_{\rho'}(1)^{c}} f \,  d\mu(s,f)}{\phi_{n}} \right) \nonumber \\
= \int_{0}^{t} \int_{B^{c}} y \,  d\xi(s,f). \label{eqEquRandoIntegCylLevy}
\end{gather} 
Then, by considering the canonical decomposition of $L(\phi_{1}, \dots, \phi_{n})$, \eqref{eqSemimartDecompNucleSpace}, \eqref{eqCharactOfLevyProcess},   
\eqref{eqEqualiContPartCylLevy},  \eqref{eqEquCompRandoIntegCylLevy}, and \eqref{eqEquRandoIntegCylLevy}, it follows that 
\begin{equation} \label{eqDefiLinPartCylLevy}
a_{t}(\omega) =\left(  \inner{A_{t}(\omega)}{\phi_{1}}, \dots, \inner{A_{t}(\omega)}{\phi_{n}} \right)= t m, \quad \forall \, \omega \in \Omega, \, t \geq 0, 
\end{equation}
where $m=(\inner{\goth{m}}{\phi_{1}}, \dots, \inner{\goth{m}}{\phi_{n}}) \in \R^{n}$. 

Now, the particular form of  the characteristics $(a,c,\lambda)$ of the $\R^{n}$-valued  semimartingale $L(\phi_{1}, \dots, \phi_{n})$ given in \eqref{eqDefiLinPartCylLevy}, \eqref{eqDefiConQuaVarPartCylLevy},  \eqref{eqDefiCompMeasuCylLevy} respectively, and from Corollary II.4.19 in \cite{JacodShiryaev}, it follows that $L(\phi_{1}, \dots, \phi_{n})$ is a $\R^{d}$-valued  L\'{e}vy process. Therefore, $L$ induces a cylindrical L\'{e}vy process in $\Phi'$. But then, Theorem 3.8 in \cite{FonsecaMora:Levy} shows that $L$ has an indistinguishable version that is a $\Phi'$-valued L\'{e}vy process. Hence, $L$ is itself a $\Phi'$-valued L\'{e}vy process.
\end{proof}

\textbf{Acknowledgements} { 
The author acknowledge The University of Costa Rica for providing financial support through the grant ``B9131-An\'{a}lisis Estoc\'{a}stico con Procesos Cil\'{i}ndricos''. Many thanks are also due to the referees for their helpful and insightful comments that contributed greatly  to improve the presentation of this manuscript.  
}

\end{document}